\pdfoutput=1
\RequirePackage{ifpdf}
\ifpdf % We are running pdfTeX in pdf mode
\documentclass[pdftex]{sigma}
\else
\documentclass{sigma}
\fi

\numberwithin{equation}{section}
\usepackage{braket}
\usepackage{mathtools}

\usepackage{thmtools}
\usepackage{thm-restate}

\newtheorem{Th}{Theorem}[section]
\newtheorem{Cor}[Th]{Corollary}
\newtheorem{lem}[Th]{Lemma}
\newtheorem{Prop}[Th]{Proposition}

\theoremstyle{definition}
\newtheorem{Def}[Th]{Definition}
\newtheorem{Rem}[Th]{Remark}

\newcommand{\e}{{\boldsymbol{e}}}

\begin{document}
\allowdisplaybreaks

\newcommand{\arXivNumber}{1807.07703}

\renewcommand{\thefootnote}{}

\renewcommand{\PaperNumber}{041}

\FirstPageHeading

\ShortArticleName{Hecke Operators on Vector-Valued Modular Forms}

\ArticleName{Hecke Operators on Vector-Valued Modular Forms\footnote{This paper is a~contribution to the Special Issue on Moonshine and String Theory. The full collection is available at \href{https://www.emis.de/journals/SIGMA/moonshine.html}{https://www.emis.de/journals/SIGMA/moonshine.html}}}

\Author{Vincent BOUCHARD~$^\dag$, Thomas CREUTZIG~$^{\dag\ddag}$ and Aniket JOSHI~$^\dag$}

\AuthorNameForHeading{V.~Bouchard, T.~Creutzig and A.~Joshi}

\Address{$^\dag$~Department of Mathematical \& Statistical Sciences, University of Alberta,\\
\hphantom{$^\dag$}~632 Central Academic Building, Edmonton T6G 2G1, Canada}
\EmailD{\href{mailto:vincent.bouchard@ualberta.ca}{vincent.bouchard@ualberta.ca}, \href{mailto:creutzig@ualberta.ca}{creutzig@ualberta.ca}, \href{mailto:asjoshi@ualberta.ca}{asjoshi@ualberta.ca}}

\Address{$^\ddag$~Research Institute for Mathematical Sciences, Kyoto University, Kyoto 606-8502, Japan}

\ArticleDates{Received September 26, 2018, in final form May 13, 2019; Published online May 25, 2019}

\Abstract{We study Hecke operators on vector-valued modular forms for the Weil representation $\rho_L$ of a lattice $L$. We first construct Hecke operators $\mathcal{T}_r$ that map vector-valued modular forms of type $\rho_L$ into vector-valued modular forms of type $\rho_{L(r)}$, where $L(r)$ is the lattice $L$ with rescaled bilinear form $(\cdot, \cdot)_r = r (\cdot, \cdot)$, by lifting standard Hecke operators for scalar-valued modular forms using Siegel theta functions. The components of the vector-valued Hecke operators $\mathcal{T}_r$ have appeared in [\textit{Comm. Math. Phys.} \textbf{350} (2017), 1069--1121] as generating functions for D4-D2-D0 bound states on K3-fibered Calabi--Yau threefolds. We study algebraic relations satisfied by the Hecke operators $\mathcal{T}_r$. In the particular case when $r=n^2$ for some positive integer $n$, we compose $\mathcal{T}_{n^2}$ with a projection operator to construct new Hecke operators $\mathcal{H}_{n^2}$ that map vector-valued modular forms of type $\rho_L$ into vector-valued modular forms of the same type. We study algebraic relations satisfied by the operators $\mathcal{H}_{n^2}$, and compare our operators with the alternative construction of Bruinier--Stein~[\textit{Math.~Z.} \textbf{264} (2010), 249--270] and Stein~[\textit{Funct.\ Approx.\ Comment.\ Math.} \textbf{52} (2015), 229--252].}


\Classification{11F25; 11F27; 17B69; 14N35}

\renewcommand{\thefootnote}{\arabic{footnote}}
\setcounter{footnote}{0}

\section{Introduction}

The intricate mathematical consistency required of physical theories often yields new, unexpected structures in mathematics. For example, it is frequently the case that observables in string theory and gauge theory must have strong invariance properties, which may be far from obvious mathematically. In many instances these invariance properties can be formulated mathematically in terms of modularity statements.

For instance, BPS degeneracies for a particular type of bound states in type IIA string theory were studied in \cite{orig}, namely vertical D4-D2-D0 bound states in K3-fibered Calabi--Yau threefolds. This problem is closely related to black hole entropy \cite{MSW, vafa} and BPS algebras \cite{HM, HM2}. Mathematically, these D4-D2-D0 bound states can be formulated in terms of generalized Donaldson--Thomas invariants \cite{orig,GS}. Physics says that the generating function for such bound states must have strong modularity properties. More precisely, it must be a vector-valued modular form for the Weil representation of a rescaled version of the lattice polarization of the underlying threefold. In~\cite{orig}, a formula for this generating function was obtained, and modularity was proved through explicit (but rather tedious) calculations. Indeed, the proof of modularity was rather technical, while the appearance of modularity hints at a deeper theory. The original motivation for the current paper is to develop a mathematical theory underlying modularity of these generating functions.

It turns out that the modularity properties of these generating functions can be understood in terms of Hecke operators on vector-valued modular forms for the Weil representation. In this paper, we construct these Hecke operators and study their algebraic properties.

A key ingredient in our construction is the Weil representation~\cite{weil}, which is a representation of the metaplectic cover of the modular group on the group algebra of the discriminant form of an even integral lattice~$L$. The Weil representation appears naturally in~\cite{orig}, but it also plays a role in various other contexts, for instance in the construction of generalized Kac--Moody algebras whose denominator identity is an automorphic product (see for example \cite{carnahan, HS, nils1, nils2, n}). A well known example of vector-valued modular forms for the Weil representation { consists of theta functions for the positive definite rank~1 lattice $\mathbb{Z}/2m\mathbb{Z}$}. The fundamental idea behind our construction is to use Siegel theta functions to lift Hecke operators on scalar-valued modular forms to Hecke operators on vector-valued modular forms for the Weil representation. We remark that in the case of positive definite lattices, Martin Raum has already studied Hecke operators between vector-valued modular forms for different Weil representations using Jacobi forms~\cite{raum2}. We generalize this to the indefinite case, and, notably, we also construct Hecke operators that map vector-valued modular forms for a given Weil representation to vector-valued modular forms of the same Weil representation. In addition, Hecke operators on rank~1 Jacobi forms were studied by Eichler and Zagier in~\cite{ez}. Our operators are a generalization of these in view of the bijective correspondence between Jacobi forms of weight~$k$ and index $m$ ($k, m \in \mathbb{N}$) to vector-valued modular forms of weight $k-\frac{1}{2}$ for the Weil representation of the lattice $\big(\mathbb{Z}, q(x)=-mx^2\big)$ (see \cite[Chapter~2]{ez}).

We note that alternative constructions of Hecke operators on vector-valued modular forms already exist in the literature \cite{ajouz, bs,HW, raum2, rossler,stein, wer,raum}, but our construction is more general and, arguably, rather straightforward. It may also be possible to generalize our construction beyond the Weil representation as we will outline in Section~\ref{s:comp}.

In any case, for completeness, in this work we also compare our construction to the alternative framework proposed by Bruinier and Stein in~\cite{bs,stein}.

Let us now summarize the main results of this paper.

\subsection{Summary of results}
Let $L$ be an even non-degenerate integral lattice of signature $(b^+,b^-)$ with bilinear form $(\cdot,\cdot)$, and $A = L'/L$ be the associated discriminant form with $\mathbb{Q}/\mathbb{Z}$-valued quadratic form $q(\cdot) = \frac{1}{2} (\cdot, \cdot)$. We denote by $L(r)$ the lattice $L$ with the rescaled bilinear form $(\cdot, \cdot)_r =r(\cdot,\cdot)$, and by $A(r) = L'(r) / L(r)$ its associated discriminant form, with $\mathbb{Q}/\mathbb{Z}$-valued rescaled quadratic form $q_r(\cdot)=\frac{1}{2} (\cdot, \cdot)_r$.

\subsubsection{Hecke operators between Weil representations}

Let $\{ e_\lambda \}_{\lambda \in A}$ be the standard basis for the vector space $\mathbb{C}[A]$, and $\psi(\tau)=\sum\limits_{\lambda \in A}\psi_\lambda (\tau) \, e_{\lambda}$ be vector-valued modular of weight $(v, \bar v)$ for the Weil representation $\rho_L$ associated to $L$. Our first result is the construction of Hecke operators $\mathcal{T}_r$ that map vector-valued modular forms\footnote{By vector-valued modular forms here and in the rest of the introduction we simply mean $\mathbb{C}[A]$-valued real analytic functions that transform as vector-valued modular forms under the Weil representation~-- see Definition~\ref{d:vvmf}. As explained in Remark \ref{r:vvmf}, we do not impose a growth condition, or holomorphicity (meromorphicity) at the cusps, or some condition involving the Laplacian. We also include ``Jacobi-like'' variables in the definition~-- see Remark~\ref{r:jacobi}.} of type $\rho_L$ to vector-valued modular forms of type~$\rho_{L(r)}$. These Hecke operators are defined by (Definition~\ref{d:hecke1})
\begin{gather*}
\mathcal{T}_r[\psi](\tau) =
r^{w+\bar{w}- 1} \sum_{\mu \in A(r)} \Bigg( \sum_{\substack{k, l >0 \\ k l = r}} \frac{1}{l^{w + \bar w }} \sum_{s=0}^{l-1} \Delta_r(\mu,k) \e\left( - \frac{s}{k} q_r (\mu) \right) \psi_{l \mu} \left( \frac{k \tau + s}{l} \right) \Bigg) e_\mu,
\end{gather*}
where $(w, \bar w) = \big( v + \frac{b^+}{2}, \bar v + \frac{b^-}{2} \big)$, $\e(x)=\exp(2 \pi {\rm i} x)$, and
\begin{gather*}
\Delta_r(\mu,k) = \begin{cases}
1 & \text{if $\mu \in A(l) \subseteq A(r)$,}\\
0 & \text{otherwise.}
\end{cases}
\end{gather*}

The idea behind the construction is to pair the components of the vector-valued modular form $\psi(\tau)$ with the components of Siegel theta functions to construct a scalar-valued modular form, and then apply the standard Hecke operators for scalar-valued modular forms to define our Hecke operators on vector-valued modular forms appropriately. More precisely, let us define an inner product on~$\mathbb{C}[A]$ by
\begin{gather*}
\Braket{\sum_{\lambda \in A} f_\lambda e_\lambda, \sum_{\delta \in A} g_\delta e_\delta} = \sum_{\lambda \in A} f_\lambda \bar{g}_{\lambda}.
\end{gather*}
We then prove that (Theorem~\ref{t:hecke})
\begin{gather*}
T_r\left[ \Braket{\psi, \Theta_L } \right](\tau, \alpha, \beta) = \Braket{\mathcal{T}_r[\psi], \Theta_{L(r)} } (\tau, \alpha, \beta),
\end{gather*}
where $\Theta_L(\tau,\alpha,\beta)$ is the Siegel theta function of the lattice $L$, and $T_r$ are the usual Hecke operators for scalar-valued modular forms. From this relation it follows that, indeed, $\mathcal{T}_r[\psi](\tau)$ is vector-valued modular of type $\rho_{L(r)}$ and weight $(v,\bar v)$. We note that this theorem is a~gene\-ra\-li\-zation to lattices of indefinite signature of a result by Martin Raum~\cite[Proposition~5.3]{raum2}.

Let us remark that the components of $\mathcal{T}_r[\psi](\tau)$ are precisely the generating functions $Z_{r,\delta}$ of D4-D2-D0 bound states (mathematically, generalized Donaldson--Thomas invariants) on K3-fibered Calabi--Yau threefolds studied in~\cite{orig}. Therefore, an immediate corollary of our construction is vector-valued modularity of these generating functions, which was proved by direct calculations in~\cite{orig}.

Our next step is to study algebraic relations satisfied by the operators $\mathcal{T}_r$. To this end we define scaling operators $\mathcal{U}_{n^2}$ on vector-valued modular forms of type $\rho_L$ (Definition~\ref{d:scaling}):
\begin{gather*}
\mathcal{U}_{n^2}[\psi](\tau) = \sum_{\nu \in A(n^2)}\Delta_{n^2}(\nu, n) \psi_{n \nu}(\tau) e_\nu.
\end{gather*}
These are appropriate scaling operators since (Lemma~\ref{l:scaling}):
\begin{gather*}
U_{n^2}\left[ \Braket{\psi, \Theta_L} \right] (\tau, \alpha, \beta) = \Braket{ \mathcal{U}_{n^2}[\psi], \Theta_{L(n^2)}} (\tau, \alpha, \beta),
\end{gather*}
where $U_{n^2}[f](\tau,\alpha,\beta) = f(\tau, n \alpha, n\beta)$ are the standard scaling operators for scalar-valued modular forms. Then we show that (Theorem~\ref{t:heckealg1}):
\begin{itemize}\itemsep=0pt
\item for $m$ and $n$ such that $\gcd(m,n)=1$,
\begin{gather*}
\mathcal{T}_m \circ \mathcal{T}_n = \mathcal{T}_{mn};
\end{gather*}
\item for $l \geq 2$ and $p$ prime,
\begin{gather*}
\mathcal{T}_{p^{l}} = \mathcal{T}_{p} \circ \mathcal{T}_{p^{l-1}} - p^{w + \bar w -1} \mathcal{U}_{p^2} \circ \mathcal{T}_{p^{l-2}} .
\end{gather*}
\end{itemize}
Those properties are analogous to the algebraic relations satisfied by the scalar-valued Hecke operators~$T_r$.

\subsubsection{Hecke operators on the Weil representation}

We then focus on the special case when $r=n^2$ for some integer $n$. In this case, we show (Lemma~\ref{l:subrep}) that $\rho_L$ is a sub-representation of the Weil representation $\rho_{L(n^2)}$ for the rescaled lattice $L(n^2)$. This allows us to define projection operators $\mathcal{P}_{n^2}$ (Definition~\ref{d:projection}), which take vector-valued modular forms of type $\rho_{L(n^2)}$ into vector-valued modular forms of type $\rho_L$ of the same weight. These projection operators act as left inverses of the scaling operators (Lemma~\ref{l:inverses}):
\begin{gather*}
\mathcal{P}_{n^2} \circ \mathcal{U}_{n^2} = \mathcal{I}.
\end{gather*}

These projection operators allow us to define new Hecke operators $\mathcal{H}_{n^2}$ which map vector-valued modular forms of type $\rho_L$ into vector-valued modular forms of the same type and weight (Definition~\ref{d:hecke2}):
\begin{gather*}
\mathcal{H}_{n^2} = \mathcal{P}_{n^2} \circ \mathcal{T}_{n^2}.
\end{gather*}
The explicit expression for $\mathcal{H}_{n^2}$ is given by (Proposition~\ref{p:Ytau}):
\begin{gather*}
\mathcal{H}_{n^2}[\psi](\tau) =n^{2(v+\bar v-1)} \sum_{\lambda \in A} \Bigg(\sum_{\substack{\gamma \in A(n^2)\\ n \gamma = \lambda}} \sum_{\substack{k, l >0 \\ k l = n^2}} \frac{1}{l^{v + \bar v + \frac{1}{2} \dim(L) }} \nonumber \\
\hphantom{\mathcal{H}_{n^2}[\psi](\tau) =}{} \times \sum_{s=0}^{l-1} \Delta_{n^2}(\gamma,n) \Delta_{n^2}(\gamma,k) \e\left( - \frac{s}{k} q_{n^2} (\gamma) \right) \psi_{l \gamma} \left( \frac{k \tau + s}{l} \right) \Bigg) e_\lambda.
\end{gather*}

As for $\mathcal{T}_{n^2}$, we study algebraic relations satisfied by the $\mathcal{H}_{n^2}$. We obtain (Theorem~\ref{t:Salgebra}):
\begin{itemize}\itemsep=0pt
\item for $m$ and $n$ such that $\gcd(m,n)=1$,
\begin{gather*}
\mathcal{H}_{m^2} \circ \mathcal{H}_{n^2} = \mathcal{H}_{m^2 n^2};
\end{gather*}
\item for $l \geq 2$ and $p$ prime,
\begin{gather*}
\mathcal{H}_{p^{2l}}=\mathcal{P}_{p^{2l-2}} \circ \mathcal{H}_{p^2} \circ \mathcal{H}_{p^{2l-2}} \circ \mathcal{U}_{p^{2l-2}} - p^{w+\bar w-1} \mathcal{H}_{p^{2l-2}} - p^{2(w+\bar w-1)} \mathcal{H}_{p^{2l-4}}.
\end{gather*}
\end{itemize}
The recursion relation is slightly different from the standard one for scalar-valued Hecke ope\-ra\-tors. This is due to two reasons: first, $\mathcal{H}_{r}$ is only defined when $r=n^2$, and second, the projection operators $\mathcal{P}_{n^2}$ and Hecke operators $\mathcal{T}_{m^2}$ only commute when $m$ and $n$ are coprime (Lemma~\ref{l:commPT}).

\subsection{Comparison to other constructions}

\subsubsection{Comment on the relation to the work of Eichler--Zagier}
Eichler and Zagier study the space of rank~1 Jacobi forms of weight $k$ and index $m$ denoted by~$J_{k,m}$ in~\cite{ez}. In particular, they construct Hecke operators $U_l$, $V_l$, $T_l$ that map the space~$J_{k,m}$ to~$J_{k,ml^2}$,~$J_{k,ml}$ and~$J_{k,m}$ respectively for $k,l,m \in \mathbb{N}$. These parallel the Hecke operators discussed in this paper and we will point out some of these connections in Sections~\ref{hecke} and~\ref{s:n2}. Our Hecke operators~$\mathcal{T}_{r}$ and $\mathcal{H}_{n^2}$ are maps between vector-valued modular forms for the Weil representation of lattices related by a rescaling or between vector-valued modular forms for the Weil representation of the same lattice. In the rank 1 case, these behave like operators that multiply and preserve the index respectively. In addition, several of the algebraic relations between Hecke operators in this article have analogues in the work of Eichler--Zagier.

\subsubsection{Comparison to the work of Bruinier and Stein \cite{bs,stein}}\label{s:comp}

Hecke operators that map vector-valued modular forms of type $\rho_L$ into vector-valued modular forms of the same type and weight were also constructed by Bruinier and Stein in~\cite{bs,stein}. The approach however is quite different. In \cite{bs} the authors first construct Hecke operators $T_{m^2}^{\rm (BS)}$ where $m$ is a positive integer that is coprime with the level $N$ of the lattice $L$. They do so by extending the Weil representation of $Mp_2(\mathbb{Z})$ to some appropriate subgroup of $\widetilde{\mathrm{GL}}_2^+(\mathbb{Q})$. They then extend their construction to Hecke operators $T_{m^2}^{\rm (BS)}$ for all positive integers $m$. However, explicit formulae are only given when $m$ is coprime with the level of the lattice. Stein generalizes this in~\cite{stein} by providing the explicit action of their Hecke operators $T_{p^{2l}}^{\rm (BS)}$ for any odd prime $p$ and positive number~$l$.

Given that the construction of Bruinier and Stein is \emph{a priori} quite different from ours, it is interesting to compare the two and investigate whether the resulting Hecke operators $T_{p^{2l}}^{\rm (BS)}$ and~$\mathcal{H}_{p^{2l}}$ are the same. In Section~\ref{form2}, we prove a precise match between our Hecke operators and the Bruinier--Stein Hecke operators. More precisely, we get an exact match only after fixing a calculational mistake in~\cite{stein}. We believe that there is a mistake in the statement and proof of Theorem~5.2 of~\cite{stein} that provides explicit formulae for their extension of the Weil representation. We redid the calculation and obtained slightly different formulae. For completeness, we present our derivation in Appendix~\ref{a:derivation}. We get an exact match with the Bruinier--Stein Hecke operators only when we use the alternative formulae for their extension of the Weil representation that we derive in Appendix~\ref{a:derivation}.

While our Hecke operators match with the Bruinier--Stein Hecke operators, we note however that our construction is fairly straightforward and more general. For instance, our Hecke operators are constructed for any $r$. But perhaps more interestingly, our construction should generalize beyond the Weil representation: it should apply whenever one has a pairing of two vector-valued modular forms that yield a scalar-valued modular form, to which one can apply standard Hecke operators. The key is to choose one of the two vector-valued modular forms carefully so that we know how it transforms under the action of ${\rm GL}_2^+(\mathbb{Q})$. In the case of the Weil representation, this was accomplished by using Siegel theta functions for the pairing.

In particular this could also be done for representations $\rho$ whose kernel contains a principal congruence subgroup (called congruence representations in literature). In this case, it is possible to embed $\rho$ in a Weil representation~$\rho_L$ associated to a lattice~$L$ (see~\cite{ehol}) and apply the construction in this paper by pairing it with `dual objects' written in terms of Siegel theta functions of $L$. However, the details remain to be worked out.

\looseness=-1 But pairings of vector-valued modular forms are standard in rational conformal field theory. For example, the Hilbert space of a full rational conformal field theory is a module for two commuting rational vertex algebras and its character is given by the pairing of the character vectors of the two vertex algebras. It may then be possible to apply our construction in these cases as well, which might actually be an interesting connection to recent results of Harvey and Wu.

\subsubsection{Comment on the recent work of Harvey and Wu \cite{HW}}

\looseness=-1 Very recently Harvey and Wu proposed a construction of Hecke operators for vector-valued modu\-lar forms of the type that appear as characters of rational conformal field theories. A~rational conformal field theory corresponds to a strongly rational vertex operator algebra, that is a vertex algebra whose category of grading restricted weak modules is a modular tensor cate\-go\-ry~\cite{Huang}. The linear span of one-point functions of these modules is then a vector-valued modular form~\cite{Zhu}. Harvey and Wu's Hecke operators act on such vector-valued modular forms; in the examples that they consider, they map character vectors of a given vertex algebra to character vectors of another vertex algebra. The involved tensor categories are Galois conjugates of each other.

We do not compare our results to these recent findings. But we would like to make a brief comment. Our strategy is to first pair a~vector-valued modular form with a dual one to get a~scalar-valued one, then apply standard Hecke operators to this object, and then somehow go back to vector-valued modular forms. This procedure also has a nice vertex algebra perspective.
Assume that you have two strongly rational vertex algebras $V$ and $W$ with modular tensor categories~$\mathcal C$ and $\mathcal D$, such that these categories are braid-reversed equivalent. Then the canonical algebra object (see \cite[Section~7.9]{etingof}) extends $V\otimes W$ to a larger vertex algebra $A$ \cite{CKM, HKL} that is self-dual, i.e., $A$ has only one simple module, $A$ itself, and its character is modular. Applying a~standard Hecke operator to this scalar-valued modular form gives another scalar-valued modular form. A natural question is wether this resulting modular form also corresponds to the character of a self-dual vertex algebra and if this vertex algebra is an extension of interesting subalgebras. To give a concrete example: let $V$ be the affine vertex algebra of $\mathfrak {g}_2$ at level one and $W$ the affine vertex algebra of $\mathfrak f_4$ at level one. Then both $V$ and $W$ have only two inequivalent simple objects and their modular tensor categories are braid-reversed equivalent~\cite{davydov}. The corresponding extension is nothing but the vertex algebra of the self-dual lattice~$E_8$, so that its character is~$\theta_{E_8}/\eta^8$, where $\theta_{E_8}$ is the theta function of $E_8$ and $\eta$ the Dedekind's eta-function. Harvey and Wu's Hecke operators relate the character vectors of these two vertex algebras to the ones of other vertex algebras, for example the Yang--Lee Virasoro minimal model. We aim to investigate if one can recover their findings from our perspective.

\subsection{Outline}
In Section~\ref{pre} we review basic facts pertaining to vector-valued modularity, lattices and Siegel theta functions.
In Section~\ref{hecke} we construct the Hecke operators $\mathcal{T}_r$, the scaling opera\-tors~$\mathcal{U}_{n^2}$, and study their algebraic relations.
In Section~\ref{s:n2} we focus on the particular case when $r=n^2$. We prove the existence of a sub-representation~$\rho_L$ of~$\rho_{L(n^2)}$, and construct projection opera\-tors~$\mathcal{P}_{n^2}$. We then define the Hecke opera\-tors~$\mathcal{H}_{n^2}$ and study the corresponding algebraic relations.
Finally, in Section~\ref{form2} we compare our Hecke operators~$\mathcal{H}_{n^2}$ with those of Bruinier and Stein from~\cite{bs,stein}. To this end, we provide an alternative calculation of the extension of the Weil representation studied in \cite{bs,stein} in Appendix~\ref{a:derivation}. The resulting formulae should replace those in the statement of Theorem~5.2 of~\cite{stein}.

\section{Preliminaries}\label{pre}

\subsection{Vector-valued modularity}

Let us start by introducing functions that are vector-valued modular. We follow the approach of Borcherds~\cite{Bo,bo1}.

Let $\tau = x+{\rm i}y \in \mathbb{H} = \{ \tau \in \mathbb{C} \,|\, \mathrm{Im}(\tau) >0 \}$, and $M=\begin{psmallmatrix}
a & b \\ c & d
\end{psmallmatrix} \in \mathrm{SL}_2(\mathbb{Z})$. We define the action of~$M$ on $\tau$ by
\begin{gather*}
M\colon \ \tau \mapsto M \tau = \frac{a \tau + b}{c \tau + d}.
\end{gather*}

The double cover of $\mathrm{SL}_2(\mathbb{Z})$ is called the metaplectic group, and is denoted by $\mathrm{Mp}_2(\mathbb{Z})$. It consists of pairs $(M, \phi_M(\tau))$, where $M = \begin{psmallmatrix} a & b \\ c & d \end{psmallmatrix} \in \mathrm{SL}_2(\mathbb{Z})$ and $\phi_M(\tau)$ is a holomorphic function on the upper half-plane $\mathbb{H}$ such that $\phi_M(\tau)^2 = c \tau + d$. The group multiplication law is given by
\begin{gather*}
(M_1, \phi_{M_1}(\tau) ) \cdot (M_2, \phi_{M_2}(\tau) ) = (M_1 M_2, \phi_{M_1}(M_2 \tau) \phi_{M_2}(\tau) ).
\end{gather*}
$\mathrm{Mp}_2(\mathbb{Z})$ is generated by
\begin{gather*}
T=\left(\begin{pmatrix}
1 &1 \\ 0 & 1
\end{pmatrix},1\right) \qquad \text{and} \qquad S=\left(\begin{pmatrix}
0 &-1 \\ 1 & 0
\end{pmatrix},\sqrt{\tau}\right).
\end{gather*}

Let $\rho$ be a representation of $\mathrm{Mp}_2(\mathbb{Z})$ on some vector space $V$, and let $W$ be a $\mathbb{R}$-vector space.

\begin{Def}\label{d:vvmf}For $v, \bar v \in \frac{1}{2} \mathbb{Z}$, we say that a $V$-valued real analytic function $\psi(\tau,\alpha,\beta)$ on $\mathbb{H} \times W \times W$ is \emph{vector-valued modular of weight $(v,\bar v)$ and type $\rho$} if
\begin{gather*}
\psi(M \tau, a \alpha + b \beta, c \alpha + d \beta) = \phi_M(\tau)^{2 v} \overline{\phi_M(\tau)}^{2 \bar v} \rho(M,\phi) \psi(\tau,\alpha, \beta),
\end{gather*}
for all $(M, \phi_M) \in \mathrm{Mp}_2(\mathbb{Z})$. We say that it is \emph{scalar-valued modular} if $V$ is one-dimensional, $v, \bar v \in \mathbb{Z}$ and $\rho$ is trivial. We denote by $M_{v, \bar v, \rho}$ the space of $V$-valued real analytic functions on $\mathbb{H} \times W \times W$ that are vector-valued modular of weight $(v, \bar v)$ and type $\rho$.
\end{Def}

\begin{Rem}\label{r:vvmf}In Definition~\ref{d:vvmf} we do not impose a growth condition, or holomorphicity (meromorphicity) at the cusps, or that the functions satisfy a condition involving the Laplacian. All that we impose in this paper is the vector-valued modular transformation property {as this is all that is required for our construction}. However, our construction could potentially restrict to various classes of modular objects, such as holomorphic modular forms, weakly holomorphic modular forms, Maass forms, etc., {after checking that the Hecke operators preserve the imposed condition.}
\end{Rem}

\begin{Rem}\label{r:jacobi}Note that in Definition~\ref{d:vvmf} we include ``Jacobi-like'' variables; these are needed for our construction. But for $\alpha = \beta = 0$ we recover the standard transformation property of vector-valued modular forms. For clarity we will drop the dependence on $\alpha$ and $\beta$ when we consider objects that transform as vector-valued modular forms.
\end{Rem}

\subsection{Lattices, discriminant forms and Weil representation}

In this paper we will focus on vector-valued modularity when $\rho$ is chosen to be the Weil representation of an even integral lattice $L$.

Let $L$ be an even, non-degenerate, integral lattice of signature $(b^+, b^-)$, with $\mathrm{sgn}(L) = b^+ - b^-$ and $\dim(L) = b^+ + b^-$. We denote by $(\cdot, \cdot): L \times L \to \mathbb{Z}$ the symmetric bilinear form on $L$.

Let $L':=\mathrm{Hom}_{\mathbb{Z}}(L, \mathbb{Z})$ be the dual lattice of $L$,
\begin{gather*}
L'= \{x \in L \otimes \mathbb{Q}\,|\, (x,y) \in \mathbb{Z} \text{ for all } y \in L \}.
\end{gather*}

Since $L$ is integral we have $L \subseteq L'$. The discriminant group of $L$ is the finite abelian group $A = L' / L$. When $L$ is even we define the discriminant form $(A, q(\cdot))$ as $A$ equipped with the $\mathbb{Q} / \mathbb{Z}$-valued quadratic form
\begin{align*}
q\colon \ \qquad A &\to \mathbb{Q} / \mathbb{Z},\\
x + L &\mapsto \tfrac{1}{2} (x,x) \text{ mod $\mathbb{Z}$.}
\end{align*}
The associated bilinear form $A \times A \to \mathbb{Q} / \mathbb{Z}$ is $(x+L, y+L) \mapsto (x,y) \text{ mod $\mathbb{Z}$}$.

Let $\{ e_\gamma \}_{\gamma \in A}$ be the standard basis for the vector space $\mathbb{C}[A]$ with $e_\gamma e_\lambda = e_{\gamma + \lambda}$. We define an inner product on $\mathbb{C}[A]$ by
\begin{gather*}%\label{e:pair}
\Braket{\sum_{\lambda \in A} f_\lambda e_\lambda, \sum_{\delta \in A} g_\delta e_\delta} = \sum_{\lambda \in A} f_\lambda \bar{g}_{\lambda}.
\end{gather*}

This can be used to define a Petersson inner product (see \cite[equation~(2.15)]{bs}) on the space of holomorphic vector-valued modular forms of weight $(k,0)$ that converges when $\langle f(\tau), g(\tau) \rangle$ is a cusp form,{\samepage
\begin{gather}\label{e:petersson}
(f,g)= \int_{\mathrm{Mp}_2(\mathbb{Z})\backslash \mathbb{H}} \langle f(\tau), g(\tau) \rangle y^k \frac{{\rm d}x {\rm d}y}{y^2},
\end{gather}
where $\tau = x+{\rm i}y$.}

Every discriminant form $(A, q(\cdot))$ defines a unitary representation of the metaplectic group $\mathrm{Mp}_2(\mathbb{Z})$ on~$\mathbb{C}[A]$:

\begin{Def}\label{d:weil}The \emph{Weil representation} $\rho_L$ of $\mathrm{Mp}_2(\mathbb{Z})$ on $\mathbb{C}[A]$ is defined by
\begin{gather*}%\label{weil}
\rho_L(T)e_{\lambda} = \e(q(\lambda))\, e_{\lambda}, \\
\rho_L(S)e_{\lambda} =\frac{\e(-\mathrm{sgn}(L)/8)}{\sqrt{|A|}} \sum_{\mu \in A} \e(-(\lambda, \mu))\,e_{\mu},
\end{gather*}
where $S$ and $T$ are the generators of $\mathrm{Mp}_2(\mathbb{Z})$. Here, we introduced the abbreviation $\e(x)=\exp(2 \pi{\rm i} x)$, which will be used throughout the paper.
\end{Def}

It is easy to see that the Weil representation is unitary with respect to the inner product:
\begin{gather}\label{eq:unitary}
\Braket{\rho_L(M,\phi_M) e_\lambda, \rho_L(M, \phi_M) e_\beta} = \Braket{e_\lambda, e_\beta} = \delta_{\lambda \beta},
\end{gather}
for all $(M, \phi_M) \in \mathrm{Mp}_2(\mathbb{Z})$ and $\lambda, \beta \in A$. Here, $\delta_{\lambda \beta}$ is the Kronecker delta, which is $1$ if $\lambda = \beta$ and $0$ otherwise.

Given an even non-degenerate lattice $L$, and its discriminant form $A$, we can thus consider real analytic functions that are vector-valued modular of type $\rho_L$, with $\rho_L$ the Weil representation of $L$.
We denote by $M_{v, \bar v, L} := M_{v, \bar v, \rho_L}$ the space of $\mathbb{C}[A]$-valued real analytic functions on $\mathbb{H} \times W \times W$, where $W = L \otimes \mathbb{R}$, that are vector-valued modular of weight $(v,\bar v)$ and type $\rho_L$.

In this paper we will also consider lattice rescalings. Let $r$ be a positive integer. We denote by $L(r)$ the lattice $L$ but with rescaled bilinear form $(\cdot, \cdot)_r := r (\cdot, \cdot)$. Let $L(r)'$ be its dual lattice, which is defined as usual by
\begin{gather*}
L(r)' = \{ x \in L \otimes \mathbb{Q} \,|\, (x,y)_r \in \mathbb{Z} \text{ for all } y \in L \}.
\end{gather*}
By definition, $L(r)' = \frac{1}{r} L'$, and thus $L' \subseteq L(r)'$. We denote the rescaled discriminant form by $A(r) = L(r)'/ L(r) \cong \frac{1}{r} L'/L$. Hence $A \subseteq A(r)$. The induced quadratic form is:
\begin{align*} \nonumber
q_r\colon \ \quad A(r) &\to \mathbb{Q} / \mathbb{Z}, \\
x + L &\mapsto \tfrac{1}{2} (x,x)_r \text{ mod $\mathbb{Z}$.}
\end{align*}

We also introduce the following notation, which will be useful later on:
\begin{Def}\label{d:delta}For any $\mu \in A(r)$, and positive integers $k$ and $l$ such that $k l = r$, we define $\Delta_r(\mu,k)$ by
\begin{gather*}%\label{eq:deltadefn}
\Delta_r(\mu,k) = \begin{cases}
1 & \text{if $\mu \in A(l) \subseteq A(r)$,}\\
0 & \text{otherwise.}
\end{cases}
\end{gather*}
\end{Def}

\subsection{Siegel theta functions}

Let $\mathrm{Gr}(L)$ be the Grassmannian of $L$, which is the set of positive definite $b^+$-dimensional subspaces of $L \otimes \mathbb{R}$. Let $v_+ \in \mathrm{Gr}(L)$, and $v_-$ be its orthogonal complement in $L \otimes \mathbb{R}$. For any $\lambda \in L \otimes \mathbb{R}$, we denote its projection onto the subspaces $v_{\pm}$ by $\lambda_{\pm}$.

Following Borcherds \cite{Bo}, we introduce the following definition.
\begin{Def}\label{d:siegel}
Let $\alpha, \beta \in L \otimes \mathbb{R}$.
The \emph{Siegel theta function} of a coset $L+\gamma$ of $L$ in $L'$ is given by\footnote{For simplicity we suppress the dependence on the choice of subspace $v^+ \in \mathrm{Gr}(L)$.}
\begin{gather*}
\theta_{L+\gamma}(\tau,\alpha,\beta) = \sum_{\lambda \in L + \gamma} \e\left(\tau q((\lambda+\beta)_{+}) + \bar \tau q((\lambda+\beta)_{-}) - \left(\lambda +\frac{\beta}{2},\alpha \right) \right).
\end{gather*}
We also define the $\mathbb{C}[A]$-valued function
\begin{gather*}
\Theta_L(\tau,\alpha,\beta) = \sum_{\gamma \in A} \theta_{L + \gamma}(\tau,\alpha,\beta) e_\gamma.
\end{gather*}
\end{Def}

\begin{Rem}The Siegel theta functions of Borcherds are similar to the Jacobi theta functions of a lattice $L$ with elliptic variable $z$ given by the realification $\beta \tau +\alpha$. In particular when $L$ is positive definite we have
\begin{gather*}
\theta_{L+\gamma}(\tau, \alpha, \beta) = \e(\tau q(\beta) -\left(\beta/2, \alpha \right) ) \widetilde{\theta}_{L+\gamma}(\tau, \beta \tau +\alpha ),
\end{gather*}
where
\begin{gather*}
\widetilde{\theta}_{L+\gamma}(\tau,z) = \sum_{\lambda \in L+\gamma} \e\left(\tau q(\lambda)+(\lambda,z) \right)
\end{gather*}
is the usual definition of Jacobi theta functions.
\end{Rem}

In \cite{Bo} Borcherds proved the following theorem.
\begin{Th}[{\cite[Theorem~4.1]{Bo}}]\label{t:thetaweil}
\begin{gather*}
\Theta_L(M \tau, a \alpha + b \beta, c \alpha + d \beta) = \phi(\tau)^{b^+} \overline{\phi(\tau)}^{b^-} \rho_L(M,\phi) \Theta_L(\tau,\alpha,\beta),
\end{gather*}
for all $(M,\phi) \in \mathrm{Mp}_2(\mathbb{Z})$. In other words, $\Theta_L(\tau,\alpha,\beta)$ is vector-valued modular of weight \linebreak $\big( \frac{1}{2} b^+, \frac{1}{2} b^- \big)$ and type~$\rho_L$, where $\rho_L$ is the Weil representation of $L$.
\end{Th}

Given two functions $\sum\limits_{\lambda \in A} f_\lambda(\tau) e_\lambda$ and $\sum\limits_{\lambda \in A} g_\lambda(\tau) e_\lambda$ that are vector-valued of type $\rho_L$ and weight $(v,\bar v)$ and $(w, \bar w)$ respectively, it is clear that
\begin{gather*}
\Braket{\sum_{\lambda \in A} f_\lambda(\tau) e_\lambda, \sum_{\lambda \in A} g_\lambda(\tau) e_\lambda} = \sum_{\lambda \in A} f_\lambda(\tau) \bar{g}_\lambda (\tau)
\end{gather*}
is scalar-valued of weight $(v+w, \bar v + \bar w)$, since the Weil representation is unitary with respect to the inner product, see~\eqref{eq:unitary}. But using Siegel theta functions we can also get a converse statement, which turns out to be very useful due to the linear independence of the Siegel theta functions :

\begin{lem}\label{l:scatovv}
$\psi(\tau)$ is vector-valued modular of type $\rho_L$ and weight $(v,\bar v)$ if and only if
\begin{gather*}
\Braket{ \psi, \Theta_L } (\tau, \alpha, \beta)= \sum_{\lambda \in A} \psi_\lambda(\tau) \bar \theta_{L+\lambda}(\tau, \alpha, \beta)
\end{gather*}
is scalar-valued modular of weight $(w, \bar w) = \big(v + \frac{1}{2} b^+, \bar v + \frac{1}{2} b^- \big)$.
\end{lem}

\begin{proof}On the one hand, if $\psi(\tau)$ is vector-valued of type $\rho_L$ and weight $(v, \bar v)$, then it follows directly that $\Braket{ \psi, \Theta_L } (\tau, \alpha, \beta)$ is scalar-valued of weight $\big(v + \frac{1}{2} b^+, \bar v + \frac{1}{2} b^- \big)$, since $\Theta_L(\tau, \alpha, \beta)$ is vector-valued of type $\rho_L$ and weight $\big( \frac{1}{2} b^+, \frac{1}{2} b^- \big)$ and the Weil representation is unitary with respect to the inner product (see~\eqref{eq:unitary}).

On the other hand, if $\Braket{ \psi, \Theta_L } (\tau, \alpha, \beta)$ is scalar-valued of weight $(w,\bar w)$, then $\psi(\tau)$ must be vector-valued of type $\rho_L$ and weight $(v,\bar v) = \big( w - \frac{1}{2} b^+, \bar w - \frac{1}{2} b^- \big)$. This follows again from unitary of the Weil representation, but also from the fact that the components $\bar \theta_{L+\lambda}(\tau, \alpha, \beta)$ of the Siegel theta functions are non-zero and linearly independent, which is crucial. This is why we need to include Jacobi-like variables $\alpha$ and $\beta$; otherwise the components of the Siegel theta functions would not be linearly independent in general, and we would not be able to deduce vector-valued modularity for $\psi(\tau)$ directly.
\end{proof}

\begin{Rem}\label{r:linearind}A proof of the linear independence of Jacobi theta functions by Boylan appears in \cite[Proposition 3.33]{boy}. The linear independence of Siegel theta functions (in the $\alpha$ variable) can be proved using a similar approach to Boylan's proof. For completeness, we redo the proof below in the case of Siegel theta functions.
\end{Rem}

\begin{lem}The Siegel theta functions $\{\theta_{L+\gamma}(\tau,\alpha,\beta) \}_{\gamma \in L'/L}$ are linearly independent in the~$\alpha$ variable $($that is for fixed values of~$\tau$ and~$\beta)$.
\end{lem}

\begin{proof}Fix $\tau \in \mathcal{H}$ and $\beta \in L\otimes \mathbb{R}$ and consider the linear combination
\begin{gather*} \phi(\alpha)=\sum_{\gamma \in L'/L}\phi_{\gamma} \theta_{L+\gamma}(\tau,\alpha,\beta)\end{gather*}
for some constants $\phi_{\gamma}$ in $\mathbb{C}$. From Definition~\ref{d:siegel} of the Siegel theta functions we have the property that for any $\gamma \in L'/L$
\begin{gather*}
\theta_{L+\lambda}(\tau, \alpha+ \gamma, \beta) = \e\left(-\left(\lambda + \frac{\beta}{2}, \gamma \right)\right) \theta_{L+\lambda}(\tau, \alpha, \beta).
\end{gather*}
Now for any $\lambda_0 \in L'/L$ we can do the following computation
\begin{gather*}
\sum_{\gamma \in L'/L} \phi(\alpha + \gamma) \e\left(-\left(\gamma, \lambda_0 - \frac{\beta}{2}\right)\right) \\
\qquad{} = \sum_{\gamma \in L'/L}\sum_{\lambda \in L'/L} \phi_\lambda \theta_{L+\lambda}(\tau,\alpha+\gamma,\beta)\e\left(-\left(\gamma, \lambda_0 - \frac{\beta}{2}\right)\right) \\
\qquad{} =\sum_{\gamma \in L'/L} \sum_{\lambda \in L'/L} \phi_{\lambda} \theta_{L+\lambda}(\tau,\alpha,\beta) \e\left((\gamma,\lambda - \lambda_0)\right) = \phi_{\lambda_0} |L'/L| \theta_{L+\lambda_0} (\tau,\alpha,\beta),
\end{gather*}
where in the last step we have used the property that the sum over $\lambda$ disappears unless $\lambda =\lambda_0$. The last equation above implies that if $\phi(\alpha)$ vanishes identically then $\phi_{\lambda}=0$ for all $\lambda \in L'/L$ and thus the lemma is proved.
\end{proof}

We now prove a lemma relating Siegel theta functions of $L$ and $L(r)$. This lemma will be essential in the next section for constructing our Hecke operators.

\begin{lem}\label{l:theta}Let $k$, $l$, $r$ be positive integers such that $k l = r$, and let $s \in \{0, 1, \ldots, l -1 \}$. Let $L+\gamma$ be a coset of $L$ in $L'$, with $\gamma \in A$. Then
\begin{gather*}%\label{eq:sumdelta}
\theta_{L + \gamma} \left( \frac{k \tau + s}{l },k \alpha + s \beta, l \beta \right)
= \sum_{\substack{\nu \in A(r) \\ l \nu = \gamma}} \Delta_r(\nu,k) \e\left( \frac{s}{k} q_r (\nu) \right) \theta_{L(r) + \nu}(\tau,\alpha,\beta),
\end{gather*}
where $\nu + L(r)$ is a coset of $L(r)$ in $L(r)'$, with $\nu \in A(r)$, and $\Delta_r(\mu,k)$ defined in Definition~{\rm \ref{d:delta}}.
\end{lem}

\begin{proof}Let $\lambda \in L + \gamma$, with $\gamma \in A$. First we compute that
\begin{gather*}
\theta_{L + \gamma} \left( \frac{k \tau + s}{l }, k \alpha + s \beta, l \beta\right) \\
\qquad{} = \sum_{\lambda \in L + \gamma} \e \left( \frac{k \tau + s}{l } q\left((\lambda+l \beta)_+\right) + \frac{k \bar \tau + s}{l } q\left((\lambda+l \beta)_-\right) - \left(\lambda + \frac{l \beta}{2}, k \alpha + s\beta\right) \right)\\
\qquad{} = \sum_{\lambda \in L + \gamma} \e\left( \frac{k \tau}{l} q \left(\left( \lambda + l \beta \right)_+ \right) + \frac{k \bar \tau}{l} q\left( \left( \lambda + l \beta \right)_- \right) - k \left(\lambda +\frac{l \beta}{2},\alpha\right) \right) \e\left( \frac{s}{l} q (\lambda) \right) .
\end{gather*}
Now there is a bijection between elements $\lambda$ of the coset $L + \gamma$ and elements $\delta$ of the cosets $L + \nu$, with $\nu \in A(l)$ and such that $l \nu = \gamma$. The bijection is given by lattice rescaling, that is, $\lambda \mapsto \delta = \frac{1}{l} \lambda$. We use this to rewrite the sum as follows
\begin{gather*}
\theta_{L + \gamma} \left( \frac{k \tau + s}{l },k\alpha + s \beta, l \beta \right) \\
\qquad{} = \sum_{\substack{\nu \in A(l) \\ l \nu = \gamma}} \sum_{\delta \in L + \nu} \e\left( \tau q_r\left( (\delta+\beta)_+ \right) + \bar \tau q_r\left( (\delta+\beta)_- \right) -\left(\delta +\frac{\beta}{2},\alpha\right)_r \right) \e\left( s q_l \left( \delta \right) \right)\\
\qquad{} = \sum_{\substack{\nu \in A(l)\\ l \nu = \gamma}}\e\left( s q_l (\nu) \right) \sum_{\delta \in L + \nu} \e\left( \tau q_r\left( (\delta+\beta)_+ \right) + \bar \tau q_r\left( (\delta+\beta)_- \right) -\left(\delta +\frac{\beta}{2},\alpha\right)_r \right),
\end{gather*}
where in the last line we used the fact that $q_l (\delta) = q_l (\nu)$ $\text{mod} \, \mathbb{Z}$, since $\nu \in A(l)$.

We now extend the sum over $\nu \in A(l) \subseteq A(r)$ to a sum over all elements $\nu \in A(r)$, using the Delta function from Definition~\ref{d:delta}. We get
\begin{gather*}
\theta_{L + \gamma} \left( \frac{k \tau + s}{l },k\alpha + s \beta, l \beta \right) = \sum_{\substack{\nu \in A(r)\\ l \nu = \gamma}} \Delta_r(\nu,k) \e\left( \frac{s}{k} q_r (\nu) \right)\theta_{L(r) + \nu}(\tau,\alpha,\beta),
\end{gather*}
where we introduced the Siegel theta functions of the rescaled lattice $L(r)$
\begin{gather*}
\theta_{L(r) + \nu}(\tau,\alpha,\beta) = \sum_{\delta \in L + \nu} \e\left( \tau q_r ( (\delta+\beta)_+ ) + \bar \tau q_r ( (\delta+\beta)_- ) -\left(\delta +\frac{\beta}{2},\alpha\right)_r \right).\tag*{\qed}
\end{gather*}\renewcommand{\qed}{}
\end{proof}

\section{Hecke operators}\label{hecke}
In this section, we define Hecke operators on $M_{v, \bar v, L}$ and study their algebraic properties.

\subsection{Classical Hecke operators}\label{scalarf}
Let us start by reviewing the standard theory of Hecke operators.
\begin{Def}Let $r$ be a positive integer and $f(\tau,\alpha,\beta)$ be scalar-valued modular of weight $(w, \bar w)$, as defined in Definition~\ref{d:vvmf}. We define the following Hecke operators on $f(\tau,\alpha,\beta)$
\begin{gather}\label{eq:heckejacobi}
T_r[f](\tau,\alpha,\beta) = r^{w + \bar{w}-1} \sum_{\substack{ k, l >0 \\ kl =r}} l ^{-w-\bar w} \sum_{s=0}^{l -1} f\left(\frac{k\tau+s}{l },k\alpha + s \beta, l \beta \right).
\end{gather}
\end{Def}
\begin{lem}\label{l:heckeclass}
$T_r[f](\tau,\alpha,\beta)$ is scalar-valued modular of weight $(w, \bar w)$.
\end{lem}
\begin{proof}
Even with the addition of Jacobi-like variables, the argument is word by word the same as for standard modular forms (see for example \cite[Proposition 2.28]{wstein}).
\end{proof}

\begin{Rem}The operators $T_r$ defined above are analogous to the Hecke operators $V_r$ of Eichler--Zagier in \cite[Section~1.4]{ez} after a certain choice of coset representatives.
\end{Rem}

\begin{Rem}\label{r:heckecoprime}We note here that there is a different definition of Hecke operators as double coset operators of the modular or the metaplectic group (see for instance~\cite{DS}). We will use this alternative definition in Section \ref{form2} in making the comparison to the work of Bruinier--Stein. More specifically, the decomposition of a double coset of the metaplectic group considered by Bruinier--Stein amounts to imposing the condition that the summation variable $s$ and $r$ are coprime (where $s$, $r$ are as in~\eqref{eq:heckejacobi}). This gives a different definition of the Hecke operators, but it is just a choice, and does not affect modularity or the algebraic results in any way.
\end{Rem}

To study algebraic relations satisfied by Hecke operators, we define scaling operators:

\begin{Def}\label{d:scalingscalar}Let $r$ be a positive integer and $f(\tau,\alpha,\beta)$ be scalar-valued modular of weight $(w,\bar w)$. We define the scaling operators $U_{r^2}$ by
\begin{gather*}
U_{r^2}[f](\tau,\alpha,\beta)=f(\tau, r\alpha, r \beta).
\end{gather*}
\end{Def}
It is clear that:
\begin{lem}$U_{r^2}[f](\tau,\alpha,\beta)$ is scalar-valued modular of weight $(w, \bar w)$.
\end{lem}

Hecke operators satisfy algebraic relations summarized in the following lemma.
\begin{lem}\label{l:heckealgebra}
For $m$ and $n$ such that $\gcd(m,n)=1$,
\begin{gather} \label{eq:jcommute}
T_m \circ T_n = T_{mn},
\end{gather}
and for $l \geq 2$ and $p$ prime,
\begin{gather}
\label{jacobirecursion}
T_{p^{l}} = T_{p} \circ T_{p^{l-1}} - p^{w + \bar w -1} U_{p^2} \circ T_{p^{l-2}} .
\end{gather}
\end{lem}

\begin{proof}Relations~\eqref{eq:jcommute} and \eqref{jacobirecursion} can be proved following the exact same steps as the proof of the respective relations for standard modular forms presented for instance in Propositions~2.28 and~2.29 of~\cite{wstein}.
\end{proof}

\subsection[Hecke operators on $M_{v, \bar v, L}$]{Hecke operators on $\boldsymbol{M_{v, \bar v, L}}$}

Let us now define Hecke operators on the space $M_{v, \bar v, L}$ of $\mathbb{C}[A]$-valued real analytic functions that are vector-valued modular of type $\rho_L$ and weight $(v, \bar v)$.

\begin{Def}\label{d:hecke1}Let $\psi(\tau)=\sum\limits_{\lambda \in A}\psi_\lambda (\tau) e_{\lambda}$ be vector-valued modular of weight $(v,\bar v)$ and type~$\rho_L$. Let $(w, \bar w) = \big( v + \frac{b^+}{2}, \bar v + \frac{b^-}{2} \big)$. We define the operators $\mathcal{T}_r$ by
\begin{gather*}
\mathcal{T}_r[\psi](\tau) =
r^{w+\bar{w}- 1} \sum_{\mu \in A(r)} \left( \sum_{\substack{k, l >0 \\ k l = r}} \frac{1}{l^{w + \bar w }} \sum_{s=0}^{l-1} \Delta_r(\mu,k) \e\left( - \frac{s}{k} q_r (\mu) \right) \psi_{l \mu} \left( \frac{k \tau + s}{l} \right) \right) e_\mu,
\end{gather*}
with $\Delta_r(\mu,k)$ defined in Definition~\ref{d:delta}.
\end{Def}

The main result is:
\begin{Th}\label{t:hecke}For any positive integer $r$
\begin{gather*}
T_r\left[ \Braket{\psi, \Theta_L } \right](\tau, \alpha, \beta) = \Braket{\mathcal{T}_r[\psi], \Theta_{L(r)} } (\tau, \alpha, \beta).
\end{gather*}
In other words, the standard Hecke transforms of the scalar-valued $ \Braket{ \psi, \Theta_L }(\tau, \alpha, \beta)$ are equal to the scalar-valued $ \Braket{ \mathcal{T}_r[\psi], \Theta_{L(r)} } (\tau, \alpha, \beta)$ obtained by pairing $\mathcal{T}_r[\psi](\tau)$ with the Siegel theta functions of the rescaled lattice~$L(r)$.
\end{Th}

An immediate corollary, using Lemmas~\ref{l:scatovv} and~\ref{l:heckeclass}, is
\begin{Cor}\label{c:modularity}
If $\psi(\tau)$ is vector-valued modular of weight $(v,\bar v)$ and type $\rho_L$, then $\mathcal{T}_r[\psi](\tau)$
is vector-valued modular of type $\rho_{L(r)}$ of the same weight. In other words, Definition~{\rm \ref{d:hecke1}} gives Hecke operators
\begin{gather*}
\mathcal{T}_r \colon \ M_{v, \bar v, \rho_L} \to M_{ v, \bar v, \rho_{L(r)}}.
\end{gather*}
\end{Cor}

This is the main reason for Definition~\ref{d:hecke1}. Let us now prove Theorem~\ref{t:hecke}.

\begin{proof}[Proof of Theorem~\ref{t:hecke}]We have
\begin{align*}
T_r\left[ \Braket{ \psi, \Theta_L} \right](\tau, \alpha, \beta) & = T_r \left[ \sum_{\lambda \in A} \psi_{\lambda}(\tau) \bar \theta_{L+\lambda}(\tau, \alpha, \beta) \right] \\
& = r^{w+\bar{w} - 1} \sum_{\substack{k, l >0 \\ k l = r}} \frac{1}{l^{w + \bar w }} \sum_{s=0}^{l-1} \sum_{\lambda \in A} \psi_\lambda \!\left( \frac{k \tau + s}{l} \right)\! \bar \theta_{L+\lambda}\!\left( \frac{k \tau + s}{l}, k \alpha + l \beta, l \beta \right).
\end{align*}
By Lemma~\ref{l:theta}, we know that
\begin{gather*}
\bar \theta_{L + \lambda} \left( \frac{k \tau + s}{l },k \alpha + l \beta, l \beta \right)
= \sum_{\substack{\nu \in A(r) \\ l \nu = \lambda}}\Delta_r(\nu,k) \e\left( -\frac{s}{k} q_r (\nu) \right) \bar \theta_{L(r) + \nu}(\tau,\alpha,\beta).
\end{gather*}
Substituting, we get
\begin{gather*}
T_r \left[ \Braket{ \psi, \Theta_L } \right](\tau, \alpha, \beta)\nonumber\\
\quad{}= r^{w+\bar{w}- 1} \sum_{\substack{k, l >0 \\ k l = r}} \frac{1}{l^{w + \bar w }} \sum_{s=0}^{l-1} \sum_{\lambda \in A} \sum_{\substack{\nu \in A(r) \\ l \nu = \lambda}}\Delta_r(\nu,k) \e\left( - \frac{s}{k} q_r (\nu) \right) \psi_\lambda \left( \frac{k \tau + s}{l} \right) \bar \theta_{L(r) + \nu}(\tau,\alpha,\beta) \nonumber\\
\quad{}= r^{w+\bar{w}- 1}\sum_{\nu \in A(r)} \sum_{\substack{k, l >0 \\ k l = r}} \frac{1}{l^{w + \bar w }} \sum_{s=0}^{l-1} \Delta_r(\nu,k) \e\left( - \frac{s}{k} q_r (\nu) \right) \psi_{l \nu} \left( \frac{k \tau + s}{l} \right) \bar \theta_{L(r) + \nu}(\tau,\alpha,\beta)\nonumber\\
\quad{}= \Braket{\mathcal{T}_r[\psi], \Theta_{L(r)}} (\tau, \alpha, \beta),
%\label{eq:hecke}
\end{gather*}
where we used Definition~\ref{d:hecke1}.
\end{proof}

\begin{Rem}The components of the vector-valued modular $\mathcal{T}_r[\psi](\tau)$ are precisely the $Z_{r,\delta}(\tau)$ constructed in \cite[Section~6]{orig}\footnote{We leave it as an exercise for the reader to translate the notation currently used into the notation of~\cite{orig}.}, which arise naturally from the partition function of generalized Donaldson--Thomas invariants of K3-fibered Calabi--Yau threefolds. In~\cite{orig}, the relevant lattice~$L$ has rank $l$ and signature $(1,l-1)$. Thus the Siegel theta function $\Theta_L(\tau,\alpha,\beta)$ has weight $\big(\frac{b^+}{2}, \frac{b^-}{2}\big) = \big( \frac{1}{2}, \frac{l-1}{2} \big)$. The construction of~\cite{orig} starts with a vector-valued modular form~$\psi(\tau)$ of type $\rho_L$ and weight $(v,\bar v) = \big({-} 1 - \frac{l}{2}, 0 \big)$. Then it is proved by direct calculations that the $Z_{r,\delta}(\tau)$ are the components of a vector-valued modular form of the same weight and type $\rho_{L(r)}$. With the construction proposed in the current paper, such a modularity statement follows directly from Corollary~\ref{c:modularity}.
\end{Rem}

\subsection[Algebraic relations satisfied by the operators $\mathcal{T}_r$]{Algebraic relations satisfied by the operators $\boldsymbol{\mathcal{T}_r}$}

In this section we study algebraic relations satisfied by the Hecke operators $\mathcal{T}_r$. Those trickle down from the corresponding relations stated in Lemma~\ref{l:heckealgebra} for the standard Hecke operators~$T_r$.

Firstly, from Theorem~\ref{t:hecke} and the commutativity of the scalar-valued Hecke operators it immediately follows that the vector-valued Hecke operators~$\mathcal{T}_m$ commute under the coprime condition.

\begin{lem}\label{l:coprime}
For $m$ and $n$ coprime, we have
\begin{gather*}
\mathcal{T}_m \mathcal{T}_n = \mathcal{T}_n \mathcal{T}_m.
\end{gather*}
\end{lem}

Recall the scaling operators $U_{n^2}$ from Definition~\ref{d:scalingscalar}. We now define scaling operators $\mathcal{U}_{n^2}$ on~$M_{v, \bar v, L}$.

\begin{Def}\label{d:scaling}Let $\psi(\tau)=\sum\limits_{\lambda \in A}\psi_\lambda (\tau) \, e_{\lambda}$ be vector-valued modular of type $\rho_L$. We define the scaling operators $\mathcal{U}_{n^2}$ by
\begin{gather*}
\mathcal{U}_{n^2}[\psi](\tau) = \sum_{\nu \in A(n^2)}\Delta_{n^2}(\nu, n) \psi_{n \nu}(\tau) e_\nu.
\end{gather*}
\end{Def}

\begin{Rem}The scaling operator appears previously in~\cite{bru} and~\cite{n} as induction of vector-valued modular forms from isotropic subgroups $H \subset A$ of discriminant forms denoted by $g{\uparrow}^A_H$ and as the $U_n$-operator on rank~1 Jacobi forms in~\cite{ez}.
\end{Rem}

Then we have:
\begin{lem}\label{l:scaling}
For any positive integer $n$,
\begin{gather*}
U_{n^2}\left[ \Braket{\psi, \Theta_L} \right] (\tau, \alpha, \beta) = \Braket{ \mathcal{U}_{n^2}[\psi], \Theta_{L(n^2)}} (\tau, \alpha, \beta).
\end{gather*}
\end{lem}

\begin{proof}We have
\begin{gather*}
U_{n^2}\left[ \Braket{\psi, \Theta_L} \right] (\tau, \alpha, \beta) = U_{n^2} \bigg[ \sum_{\lambda \in A} \psi_\lambda(\tau) \bar \theta_{L+\lambda}(\tau, \alpha, \beta) \bigg]
 = \sum_{\lambda \in A}\psi_{\lambda}(\tau)\bar \theta_{L+\lambda}(\tau,n \alpha,n \beta).
\end{gather*}
But Lemma~\ref{l:theta}, with $k=n$, $l=n$ and $s=0$, states that
\begin{gather*}%\label{eq:theta}
\bar \theta_{L+\lambda}(\tau,n \alpha,n \beta)=\sum_{\substack{ \nu \in A(n^2) \\ n \nu=\lambda}} \Delta_{n^2}(\nu, n) \bar \theta_{L(n^2)+\nu}\left(\tau,\alpha,\beta \right).
\end{gather*}
Thus
\begin{align*}
U_{n^2}\left[ \Braket{\psi, \Theta_L} \right] (\tau, \alpha, \beta)& = \sum_{\lambda \in A}\psi_{\lambda}(\tau)\sum_{\substack{ \nu \in A(n^2) \\ n \nu=\lambda}} \Delta_{n^2}(\nu, n) \bar \theta_{L(n^2)+\nu}\left(\tau,\alpha,\beta \right)\\
&= \sum_{\nu \in A(n^2)}\Delta_{n^2}(\nu, n) \psi_{n \nu}(\tau) \bar \theta_{L(n^2)+\nu}\left(\tau,\alpha,\beta \right)\\
&= \Braket{ \mathcal{U}_{n^2}[\psi], \Theta_{L(n^2)} } (\tau, \alpha, \beta).\tag*{\qed}
\end{align*}\renewcommand{\qed}{}
\end{proof}

\begin{Rem}The proofs of Theorem~\ref{t:hecke} and Lemma~\ref{l:scaling} are analogous to the respective computations in~\cite{raum}.
\end{Rem}

It immediately follows from Lemmas \ref{l:scatovv} and \ref{l:heckeclass} that:

\begin{Cor}Let $\psi(\tau)$ be vector-valued modular of type $\rho_L$. Then $\mathcal{U}_{n^2}[\psi](\tau)$ is vector-valued modular of type $\rho_{L(n^2)}$ of the same weight. In other words, Definition~{\rm \ref{d:scaling}} gives scaling operators
\begin{gather*}
\mathcal{U}_{n^2}\colon \ M_{v,\bar v,\rho_{L}} \to M_{v, \bar v, \rho_{L(n^2)}}.
\end{gather*}
\end{Cor}

With this definition, we obtain the following theorem, analogous to Lemma~\ref{l:heckealgebra}.
\begin{Th}\label{t:heckealg1}
For $m$ and $n$ such that $\gcd(m,n)=1$,
\begin{gather*}
\mathcal{T}_m \circ \mathcal{T}_n = \mathcal{T}_{mn},
\end{gather*}
while for $l \geq 2$ and $p$ prime,
\begin{gather*}
\mathcal{T}_{p^{l}} = \mathcal{T}_{p} \circ \mathcal{T}_{p^{l-1}} - p^{w + \bar w -1} \mathcal{U}_{p^2} \circ \mathcal{T}_{p^{l-2}} .
\end{gather*}
\end{Th}

\begin{proof}These two statements follow directly by applying the analogous statements from Lem\-ma~\ref{l:heckealgebra} to the scalar-valued $\Braket{\psi, \Theta_L }(\tau, \alpha, \beta)$ and then using the definition of our operators~$\mathcal{T}_n$ and~$\mathcal{U}_{n^2}$.
\end{proof}

\section[The $r=n^2$ case]{The $\boldsymbol{r=n^2}$ case}\label{s:n2}

We now specialize to Hecke operators $\mathcal{T}_r$ with $r = n^2$ for some positive integer $n$. What is special in this case is the existence of a sub-representation $\rho_L$ of the Weil representation $\rho_{L(n^2)}$ for the rescaled lattice $L(n^2)$. In Lemma~\ref{l:theta} we wrote a formula relating Siegel theta functions of a lattice $L$ in terms of theta functions of a rescaled lattice. Using the characterization of the Weil representation as the transformation law of theta series, we give below an embedding of~$\rho_L$ into~$\rho_{L(n^2)}$ and give a proof that is independent of the Siegel theta function properties.
This allows us to define projection operators $\mathcal{P}_{n^2}$, which are left inverses of the scaling operators~$\mathcal{U}_{n^2}$. We can use these projection operators to define new Hecke operators $\mathcal{H}_{n^2} = \mathcal{P}_{n^2} \circ \mathcal{T}_{n^2}\colon M_{v, \bar v, L} \to M_{v, \bar v, L}$ which take functions that are vector-valued modular of type $\rho_L$ to functions that are vector-valued modular of the same type.

\subsection{Weil sub-representation}

Let us start by proving the existence of a sub-representation $\rho_L$ of the Weil representation $\rho_{L(n^2)}$ for the rescaled lattice $L\big(n^2\big)$. Recall from Definition~\ref{d:weil} that the Weil representation $\rho_{L(n^2)}$ of $\mathrm{Mp}_2(\mathbb{Z})$ on $\mathbb{C}\big[A\big(n^2\big)\big]$ is defined by
\begin{gather*}%\label{weil}
\rho_{L(n^2)}(T)e_{\nu} = \e(q_{n^2}(\nu)) e_{\nu}, \\
\rho_{L(n^2)}(S)e_{\nu} =\frac{\e(-\mathrm{sgn}(L)/8)}{\sqrt{|A(n^2)|}} \sum_{\mu \in A(n^2)} \e(-(\nu, \mu)_{n^2}) e_{\mu},
\end{gather*}
where $\{ e_\nu \}_{\nu \in A(n^2)}$ is the standard basis for the vector space $\mathbb{C}[A(n^2)]$, and $S$ and $T$ are the generators of~$\mathrm{Mp}_2(\mathbb{Z})$.

Consider the subspace $\mathbb{C}[A] \subseteq \mathbb{C}\big[A\big(n^2\big)\big]$ spanned by the basis vectors $\{f_\lambda \}_{\lambda \in A}$ defined by
\begin{gather*}
f_\lambda =\frac{1}{n^{\dim(L)}}\sum_{\substack{\nu \in A(n) \subseteq A(n^2) \\ n \nu = \lambda}} e_\nu.
\end{gather*}
The $\{f_\lambda \}_{\lambda \in A}$ form the standard basis for $\mathbb{C}[A]$. Indeed, one sees that $f_\lambda f_\delta = f_{\lambda +\delta}$:
\begin{align*}
f_\lambda f_\delta & = \frac{1}{n^{2 \dim(L)}}\sum_{\substack{\nu \in A(n) \\ n \nu = \lambda}} \sum_{\substack{\mu \in A(n) \\ n \mu = \delta}} e_\mu e_\nu
= \frac{1}{n^{2 \dim (L)}}\sum_{\substack{\nu \in A(n) \\ n \nu = \lambda}} \sum_{\substack{\mu \in A(n) \\ n \mu = \delta}} e_{\mu + \nu} \\
& = \frac{1}{n^{2 \dim (L)}}\sum_{\substack{\alpha \in A(n) \\ n \alpha = \lambda + \delta}} e_\alpha \left(\sum_{\substack{\mu \in A(n) \\ n \mu = \delta}} 1 \right)
= \frac{1}{n^{ \dim(L)}}\sum_{\substack{\alpha \in A(n) \\ n \alpha = \lambda + \delta}} e_\alpha
= f_{\lambda + \delta},
\end{align*}
since
\begin{gather}\label{eq:usefulsum}
\sum_{\substack{\mu \in A(n)\\ n \mu = \delta}} 1 =\left| \frac{1}{n}L/L \right|= n^{\dim (L)}.
\end{gather}

We prove the following important lemma.

\begin{lem}\label{l:subrep}The restriction of $\rho_{L(n^2)}$ to the subspace $\mathbb{C}[A]$ is the Weil representation $\rho_L$:
\begin{gather*}
\rho_{L(n^2)} \big|_{\mathbb{C}[A]} = \rho_L.
\end{gather*}
In other words,
\begin{gather*}%\label{weil}
\rho_{L(n^2)}(T)f_{\lambda} = \e(q(\lambda)) f_{\lambda} = \rho_L(T) (f_\lambda),\\
\rho_{L(n^2)}(S)f_{\lambda} =\frac{\e(-\mathrm{sgn}(L)/8)}{\sqrt{|A|}} \sum_{\gamma \in A} \e(-(\lambda, \gamma))\,f_{\lambda} = \rho_{L}(S) (f_\lambda).
\end{gather*}
\end{lem}

\begin{proof}Let us begin with the $T$ transformation
\begin{align*}
\rho_{L(n^2)}(T)f_{\lambda} &=\frac{1}{n^{\dim (L)}}\sum_{\substack{\nu \in A(n) \\ n \nu = \lambda}} \rho_{L(n^2)}(T)(e_\nu)
= \frac{1}{n^{\dim(L)}} \sum_{\substack{\nu \in A(n)\\ n \nu = \lambda}} \e(q_{n^2}(\nu))\, e_{\nu} \\
& =\frac{1}{n^{\dim(L)}}\e(q(\lambda)) \sum_{\substack{\nu \in A(n) \\ n \nu = \lambda}} e_\nu =\e(q(\lambda)) f_\lambda.
\end{align*}
As for the $S$ transformation,
\begin{align*}
\rho_{L(n^2)}(S)f_{\lambda}& =\frac{1}{n^{\dim(L)}} \sum_{\substack{\nu \in A(n) \\ n \nu = \lambda}} \rho_{L(n^2)}(S)(e_\nu) \\
& = \frac{1}{n^{\dim(L)}} \frac{\e(-\mathrm{sgn}(L)/8)}{\sqrt{|A(n^2)|}} \sum_{\substack{\nu \in A(n) \\ n \nu = \lambda}} \sum_{\mu \in A(n^2)} \e(-(\nu, \mu)_{n^2})\,e_{\mu}.
\end{align*}
Now consider the sum $\sum\limits_{\substack{\nu \in A(n) \\ n \nu = \lambda}} \e(-(\nu, \mu)_{n^2})$. We can do a shift $\nu \mapsto \nu + \beta$ for any $\beta \in \frac{1}{n} L/L$. It should not change the sum, since if $n \nu = \lambda$, then $n (\nu + \beta) = \lambda$, and hence it only amounts to relabeling the summands. Thus for all $\beta \in \frac{1}{n} L / L$, we must have:
\begin{gather*}
\sum_{\substack{\nu \in A(n) \\ n \nu = \lambda}} \e(-(\nu, \mu)_{n^2}) = \e \left( - (\beta, \mu)_{n^2} \right) \sum_{\substack{\nu \in A(n) \\ n \nu = \lambda}} \e(-(\nu, \mu)_{n^2}).
\end{gather*}
This implies that either the summation over $\nu$ is zero, or $\e \left( - (\beta, \mu)_{n^2} \right) = 1$ for all $\beta \in \frac{1}{n} L/L$, which will be the case if $\mu \in A(n) \subseteq A\big(n^2\big)$.
Thus we conclude that the summation over $\nu$ is zero whenever $\mu \notin A(n) \subseteq A\big(n^2\big)$. As a result, we get
\begin{align*}
\rho_{L(n^2)}(S)f_{\lambda} & =\frac{1}{n^{\dim(L)}}\frac{\e(-\mathrm{sgn}(L)/8)}{\sqrt{|A(n^2)|}} \sum_{\substack{\nu \in A(n) \\ n \nu = \lambda}} \sum_{\mu \in A(n)} \e(-(\nu, \mu)_{n^2})\,e_{\mu}\\ \nonumber
&=\frac{1}{n^{\dim(L)}}\frac{\e(-\mathrm{sgn}(L)/8)}{\sqrt{|A(n^2)|}} \sum_{\substack{\nu \in A(n) \\ n \nu = \lambda}} \sum_{\mu \in A(n)} \e(-(n \nu, n \mu))\,e_{\mu}\\ \nonumber
&=\frac{1}{n^{\dim(L)}}\frac{\e(-\mathrm{sgn}(L)/8)}{\sqrt{|A(n^2)|}} \left| \frac{1}{n}L / L \right| \sum_{\mu \in A(n)} \e(-(\lambda, n \mu))\,e_{\mu}\\ \nonumber
&=\frac{1}{n^{\dim(L)}}\frac{\e(-\mathrm{sgn}(L)/8)}{\sqrt{|A|}} \sum_{\delta \in A} \e(-(\lambda, \delta))\sum_{\substack{\mu \in A(n) \\ n \mu = \delta}} e_{\mu}\\
&= \frac{\e(-\mathrm{sgn}(L)/8)}{\sqrt{|A|}} \sum_{\delta \in A} \e(-(\lambda, \delta)) f_\delta.\tag*{\qed}
\end{align*}\renewcommand{\qed}{}
\end{proof}

\subsection{Projection operators}

The existence of the sub-representation given in Lemma~\ref{l:subrep} allows us to define projection ope\-ra\-tors $ \mathcal{P}_{n^2}\colon M_{v,\bar v,L(n^2)} \to M_{v, \bar v,L}$.

\begin{Def}\label{d:projection} Let $\psi(\tau)=\sum\limits_{\nu \in A(n^2)} \psi_{\nu}(\tau) \, e_{\nu}$ be vector-valued modular of type $\rho_{L(n^2)}$. We define the projection operators $\mathcal{P}_{n^2}$ by
\begin{gather*}
\mathcal{P}_{n^2}[\psi](\tau) =\frac{1}{n^{\dim(L)}} \sum_{\lambda \in A} \Bigg( \sum_{\substack{\gamma \in A(n) \\ n \gamma = \lambda}} \psi_{\gamma}(\tau)\Bigg)e_{\lambda}
=\frac{1}{n^{\dim(L)}} \sum_{\lambda \in A} \Bigg( \sum_{\substack{\gamma \in A(n^2) \\ n \gamma = \lambda}} \Delta_{n^2}(\gamma,n)\psi_{\gamma}(\tau)\Bigg)e_{\lambda},%\label{eq:pin2}
\end{gather*}
with $\Delta_{n^2}(\gamma,n)$ defined in Definition~\ref{d:delta}.
\end{Def}

\begin{Rem}\label{r:proj}The projection operator $\mathcal{P}_{n^2}$ appears in \cite[Proposition 3.2]{bru} as the `arrow-down' operator $f\downarrow^A_H$ and the `averaging operator' $ \mathcal{A}$ on rank~1 Jacobi forms in \cite[p.~51]{ez}.
\end{Rem}

As a direct corollary of Lemma~\ref{l:subrep} we get:
\begin{Cor}\label{c:projection} Let $\psi(\tau)=\sum\limits_{\nu \in A(n^2)} \psi_{\nu}(\tau) \, e_{\nu}$ be vector-valued modular of type $\rho_{L(n^2)}$. Then $\mathcal{P}_{n^2}[\psi](\tau)$ is vector-valued modular of type $\rho_{L}$ of the same weight. In other words, Definition~{\rm \ref{d:projection}} gives projection operators
\begin{gather*}
\mathcal{P}_{n^2}\colon \ M_{v,\bar v,L(n^2)} \to M_{v, \bar v,L}.
\end{gather*}
\end{Cor}

We now show that the projection operators $\mathcal{P}_{n^2}$ are left inverses of the scaling operators $\mathcal{U}_{n^2}$.

\begin{lem}\label{l:inverses}
\begin{gather*}
\mathcal{P}_{n^2} \circ \mathcal{U}_{n^2}= \mathcal{I},
\end{gather*}
where $\mathcal{I}$ is the identity operator.
\end{lem}

\begin{proof}Let $\psi(\tau)$ be vector-valued modular of type $\rho_L$. We have
\begin{align*}
\mathcal{P}_{n^2} \circ \mathcal{U}_{n^2} [\psi](\tau)& = \mathcal{P}_{n^2} \Bigg(\sum_{\nu \in A(n^2) } \Delta_{n^2}(\nu, n) \psi_{n \nu}(\tau) e_\nu \Bigg)
= \frac{1}{n^{\dim(L)}} \sum_{\lambda \in A} \Bigg(\sum_{\substack{\gamma \in A(n) \\ n \gamma = \lambda}} \psi_{n \gamma}(\tau) \Bigg) e_\lambda \\
& = \frac{1}{n^{\dim(L)}} \sum_{\lambda \in A}\Bigg(\sum_{\substack{\gamma \in A(n) \\ n \gamma = \lambda}} 1 \Bigg) \psi_\lambda(\tau) e_\lambda.
\end{align*}
The sum in bracket was evaluated in \eqref{eq:usefulsum}, and is equal to $n^{\dim(L)}$. Thus we get
\begin{gather*}
\mathcal{P}_{n^2} \circ \mathcal{U}_{n^2} [\psi](\tau)=\sum_{\lambda \in A} \psi_\lambda(\tau) e_\lambda.\tag*{\qed}
\end{gather*}\renewcommand{\qed}{}
\end{proof}

\begin{Rem}In general it is not true that
\begin{gather}\label{eq:commute2}
\mathcal{U}_{n^2} \circ \mathcal{P}_{n^2}=\mathcal{I}.
\end{gather}
However, \eqref{eq:commute2} holds for vector-valued modular forms $\psi(\tau)=\sum\limits_{\lambda \in A(n^2)}\psi_{\lambda}(\tau)e_{\lambda}$ that are not supported on $\frac{1}{n}L/L \subset A\big(n^2\big)$ (so that $\psi_{\lambda}(\tau)=0$ for $\lambda \in \frac{1}{n}L/L$). In particular if $\psi(\tau)=\sum\limits_{\lambda \in A(n^2)} \psi_{\lambda}(\tau)e_{\lambda}$ is a vector-valued modular form supported on~$A(n)$ so that $\psi_{\lambda}(\tau) =0$ whenever $\lambda \notin A(n)$, then $\mathcal{U}_{n^2} \circ \mathcal{P}_{n^2}=\mathcal{I}$. This is analogous to Proposition~3.33 in~\cite{bru}.
\end{Rem}

\subsection[Hecke operators $\mathcal{H}_{n^2}$]{Hecke operators $\boldsymbol{\mathcal{H}_{n^2}}$}

We can now compose our Hecke operators $\mathcal{T}_{n^2}$ with the projection operators $\mathcal{P}_{n^2}$ to get Hecke operators $\mathcal{H}_{n^2}\colon M_{v,\bar v,L} \to M_{v, \bar v,L}$. These mirror the Hecke operators $T_n\colon J_{k,m} \to J_{k,m}$ of Eichler--Zagier constructed as a sum over right cosets of the Jacobi group. The operators $T_n$ were also shown to be a composition of the operators $V_{n^2}\colon J_{k,m} \to J_{k,m}$ and the averaging operator $\mathcal{A}$ (see Remark~\ref{r:proj}) after imposing a certain restriction on the sum over cosets. In the following two sections, one can observe many similarities between the Hecke operators of this paper and the ones considered by Eichler--Zagier.

\begin{Def}\label{d:hecke2}We define the Hecke operators
\begin{gather*}
\mathcal{H}_{n^2}:= \mathcal{P}_{n^2} \circ \mathcal{T}_{n^2}\colon \ M_{v,\bar v,L} \to M_{v, \bar v,L}.
\end{gather*}
\end{Def}
We can give an explicit formula for the components of $\mathcal{H}_{n^2}[\psi](\tau)$.
\begin{Prop}\label{p:Ytau}
Let $\psi(\tau) = \sum\limits_{\lambda \in A} \psi_{\lambda}(\tau) e_{\lambda}$ be vector-valued modular of type $\rho_L$ and weight $(v,\bar v)$. Then $\mathcal{H}_{n^2}[\psi](\tau)$ is also vector-valued modular of type $\rho_L$ and weight $(v, \bar v)$, and can be written as
\begin{gather*}
\mathcal{H}_{n^2}[\psi](\tau)=n^{2(v+\bar v-1)}
\times \sum_{\lambda \in A} \Bigg(\sum_{\substack{\gamma \in A(n^2)\\ n \gamma = \lambda}} \sum_{\substack{k, l >0 \\ k l = n^2}} \frac{1}{l^{v + \bar v + \frac{1}{2} \dim(L) }} \\
\hphantom{\mathcal{H}_{n^2}[\psi](\tau)=}{}\times \sum_{s=0}^{l-1} \Delta_{n^2}(\gamma,n) \Delta_{n^2}(\gamma,k) \e\left( - \frac{s}{k} q_{n^2} (\gamma) \right) \psi_{l \gamma} \left( \frac{k \tau + s}{l} \right) \Bigg) e_\lambda.
\end{gather*}
\end{Prop}
\begin{proof}
This follows directly from Definitions \ref{d:hecke1} and \ref{d:projection}.
\end{proof}

\subsection[Algebraic relations satisfied by the Hecke operators $\mathcal{H}_{n^2}$]{Algebraic relations satisfied by the Hecke operators $\boldsymbol{\mathcal{H}_{n^2}}$}\label{prop}

In the previous section, we proved Theorem~\ref{t:heckealg1} for the Hecke operators $\mathcal{T}_{r}$. We now study similiar recursion relations for the operators~$\mathcal{H}_{n^2}$. The statements and their proofs in this section mirror analogous results obtained by Eichler--Zagier on Jacobi forms in \cite[Section~I.4]{ez}.

We first need the following lemmas.

\begin{lem} \label{l:commUT}For any positive integers $m$ and $n$,
\begin{gather*}
\mathcal{U}_{n^2} \circ \mathcal{T}_{m^2} = \mathcal{T}_{m^2} \circ \mathcal{U}_{n^2}.
\end{gather*}
\end{lem}

\begin{proof}This follows directly by applying the analogous statement for $U_{n^2}$ and $T_{m^2}$ on the scalar-valued $\Braket{\psi, \Theta_L }(\tau, \alpha, \beta)$ and then using the definition of our operators $\mathcal{T}_{m^2}$ and $\mathcal{U}_{n^2}$.
\end{proof}

\begin{lem}\label{l:compP}For any positive integers $m$ and $n$,
\begin{gather*}
\mathcal{P}_{m^2} \circ \mathcal{P}_{n^2} = \mathcal{P}_{m^2 n^2}.
\end{gather*}
\end{lem}

\begin{proof}Let $\psi(\tau) = \sum_{\nu \in A(m^2 n^2)} \psi_\nu (\tau) e_\nu$ be vector-valued modular of type $\rho_{L(m^2 n^2)}$. Then
\begin{align*}
\mathcal{P}_{m^2} \circ \mathcal{P}_{n^2}[\psi](\tau) &=\frac{1}{n^{\dim(L)}} \mathcal{P}_{m^2} \Bigg[\sum_{\alpha \in A(m^2)} \sum_{\substack{\gamma \in A(m^2 n^2) \\ n \gamma = \alpha}} \Delta_{m^2 n^2}(\gamma,n) \psi_\gamma(\tau) e_\alpha \Bigg]\\
&=\frac{1}{(m n)^{\dim(L)}} \sum_{\lambda \in A} \sum_{\substack{\beta \in A(m^2) \\ m \beta = \lambda}} \sum_{\substack{\gamma \in A(m^2 n^2) \\ n \gamma = \beta}} \Delta_{m^2}(\beta, m) \Delta_{m^2 n^2}(\gamma,n) \psi_{\gamma}(\tau) e_\lambda.
\end{align*}
The two delta conditions imply that $\gamma \in A(mn)$. We can then rewrite the sums as
\begin{align*}
\mathcal{P}_{m^2} \circ \mathcal{P}_{n^2}[\psi](\tau) & = \frac{1}{(m n)^{\dim(L)}} \sum_{\lambda \in A} \sum_{\substack{\beta \in A(m) \\ m \beta = \lambda}} \sum_{\substack{\gamma \in A(m n) \\ n \gamma = \beta}} \psi_\gamma(\tau) e_\lambda \\
& = \frac{1}{(m n)^{\dim(L)}} \sum_{\lambda \in A} \sum_{\substack{\gamma \in A(m n) \\ m n \gamma = \lambda}} \psi_\gamma(\tau) e_\lambda = \mathcal{P}_{m^2 n^2}[\psi](\tau).\tag*{\qed}
\end{align*}\renewcommand{\qed}{}
\end{proof}

However, the projection and Hecke operators only commute when $\gcd(m,n)=1$:

\begin{lem}\label{l:commPT}For $m$ and $n$ such that $\gcd(m,n)=1$,
\begin{gather*}
\mathcal{P}_{n^2} \circ \mathcal{T}_{m^2} = \mathcal{T}_{m^2} \circ \mathcal{P}_{n^2}.
\end{gather*}
\end{lem}

\begin{proof}We start with the left-hand side. Let $\psi(\tau)$ be vector-valued modular of weight $(v,\bar v)$ and type $\rho_{L(n^2)}$. We have
\begin{gather*}
\mathcal{P}_{n^2} \circ \mathcal{T}_{m^2}[\psi](\tau)
= m^{2(w + \bar w - 1)} \mathcal{P}_{n^2} \Bigg[ \sum_{\mu \in A(m^2 n^2)} \Bigg( \sum_{\substack{k, l >0 \\ k l = m^2}} \frac{1}{l^{w + \bar w }} \\
\hphantom{\mathcal{P}_{n^2} \circ \mathcal{T}_{m^2}[\psi](\tau) =}{}\times \sum_{s=0}^{l-1} \Delta_{m^2 n^2}(\mu,k) \e\left( - \frac{s}{k} q_{m^2 n^2} (\mu) \right) \psi_{l \mu} \left( \frac{k \tau + s}{l} \right) \Bigg) e_\mu \Bigg]\\
\hphantom{\mathcal{P}_{n^2} \circ \mathcal{T}_{m^2}[\psi](\tau)}{} = \frac{m^{2(w + \bar w - 1)} }{n^{\dim(L)}} \sum_{\lambda \in A(m^2)} \Bigg(\sum_{\substack{k, l >0 \\ k l = m^2}} \frac{1}{l^{w + \bar w }} \sum_{s=0}^{l-1} \e\left( - \frac{s}{k} q_{m^2} (\lambda) \right)\\
\hphantom{\mathcal{P}_{n^2} \circ \mathcal{T}_{m^2}[\psi](\tau) =}{} \times
\sum_{\substack{\gamma \in A(m^2 n^2) \\ n \gamma = \lambda}} \Delta_{m^2 n^2}(\gamma,n) \Delta_{m^2 n^2}(\gamma,k) \psi_{l \gamma} \left( \frac{k \tau + s}{l} \right)\Bigg) e_\lambda.
\end{gather*}
On the right-hand side, we get
\begin{gather*}
\mathcal{T}_{m^2} \circ \mathcal{P}_{n^2}[\psi](\tau) = \frac{1}{n^{\dim(L)}}\mathcal{T}_{m^2} \Bigg[\sum_{\lambda \in A} \Bigg(\sum_{\substack{\mu \in A(n^2) \\ n \mu = \lambda}} \Delta_{n^2}(\mu,n) \psi_{\mu}(\tau) \Bigg) e_\lambda \Bigg] \\
\hphantom{\mathcal{T}_{m^2} \circ \mathcal{P}_{n^2}[\psi](\tau) }{} = \frac{m^{2(w+\bar w - 1)}}{n^{\dim(L)}}\sum_{\lambda \in A(m^2)} \Bigg( \sum_{\substack{k, l >0 \\ k l = m^2}} \frac{1}{l^{w + \bar w }} \sum_{s=0}^{l-1} \e\left( - \frac{s}{k} q_{m^2} (\lambda) \right)\\
\hphantom{\mathcal{T}_{m^2} \circ \mathcal{P}_{n^2}[\psi](\tau) =}{} \times
\sum_{\substack{\mu \in A(n^2) \\ n \mu = l \lambda}} \Delta_{m^2}(\lambda,k) \Delta_{n^2}(\mu,n) \psi_{\mu}\left( \frac{k \tau + s}{l} \right) \Bigg) e_\lambda.
\end{gather*}
To prove equality between the two sides we need to show that
\begin{gather}
\sum_{\substack{\gamma \in A(m^2 n^2) \\ n \gamma = \lambda}} \Delta_{m^2 n^2}(\gamma,n) \Delta_{m^2 n^2}(\gamma,k) \psi_{l \gamma} \left( \frac{k \tau + s}{l} \right) \nonumber\\
\qquad{} = \sum_{\substack{\mu \in A(n^2) \\ n \mu = l \lambda}} \Delta_{m^2}(\lambda,k) \Delta_{n^2}(\mu,n) \psi_{\mu}\left( \frac{k \tau + s}{l} \right)\label{eq:toshow}
\end{gather}
for all $k,l>0$ such that $ k l = m^2$, $s \in \{0,\ldots, l-1 \}$, and $\lambda \in A\big(m^2\big)$.

On the right-hand side, the two delta functions impose that $\mu \in A(n)$ and $\lambda \in A(l)$, so we can write the right-hand side as
\begin{gather*}
\sum_{\substack{\mu \in A(n) \\ n \mu = l \lambda}} \psi_{\mu}\left( \frac{k \tau + s}{l} \right),
\end{gather*}
when $\lambda \in A(l)$ and zero otherwise.

On the left-hand side, the first delta function imposes that $\gamma \in A\big(m^2 n\big)$, while the second imposes that $\gamma \in A\big(l n^2\big)$. Together those impose that $\gamma \in A(s)$, where $s = \gcd\big(m^2 n, l n^2\big)$. Assuming that $\gcd(m,n)=1$, we have $s = l n$, hence $\gamma \in A(l n)$. Since $n \gamma = \lambda$, this imposes that $\lambda \in A(l) \subseteq A\big(m^2\big)$. So the left-hand side can be written as
\begin{gather*}
\sum_{\substack{\gamma \in A(l n) \\ n \gamma = \lambda}} \psi_{l \gamma} \left( \frac{k \tau + s}{l} \right),
\end{gather*}
when $\lambda \in A(l)$ and zero otherwise. We note that knowing $n \gamma$ and $l \gamma$ completely fixes $\gamma \in A(l n)$ by the Euclidean algorithm. Thus if we define $\mu = l \gamma$, we can rewrite the sum as
\begin{gather*}
\sum_{\substack{\mu \in A(n) \\ n \mu = l \lambda}} \psi_\mu \left( \frac{k \tau + s}{l} \right) ,
\end{gather*}
and \eqref{eq:toshow} is satisfied.
\end{proof}

We then prove the following algebraic relations.

\begin{Th}\label{t:Salgebra}
For $m$ and $n$ such that $\gcd(m,n)=1$,
\begin{gather*}
\mathcal{H}_{m^2} \circ \mathcal{H}_{n^2} = \mathcal{H}_{m^2 n^2},
\end{gather*}
while for $l \geq 2$ and $p$ prime,
\begin{gather*}%\label{Salgebra}
\mathcal{H}_{p^{2l}}=\mathcal{P}_{p^{2l-2}} \circ \mathcal{H}_{p^2} \circ \mathcal{H}_{p^{2l-2}} \circ \mathcal{U}_{p^{2l-2}} - p^{w+\bar w-1} \mathcal{H}_{p^{2l-2}} - p^{2(w+\bar w-1)} \mathcal{H}_{p^{2l-4}}.
\end{gather*}
\end{Th}

\begin{proof}To prove the first statement, we start with
\begin{gather*}
\mathcal{T}_{m^2} \circ \mathcal{T}_{n^2} = \mathcal{T}_{m^2 n^2},
\end{gather*}
and apply the projection operator $\mathcal{P}_{m^2 n^2} = \mathcal{P}_{m^2} \circ \mathcal{P}_{n^2}$ (using Lemma~\ref{l:compP}) on both sides of the equation. The right-hand side becomes $\mathcal{H}_{m^2 n^2}$, while the left-hand side becomes $\mathcal{H}_{m^2} \circ \mathcal{H}_{n^2}$ after using Lemma~\ref{l:commPT}.

For the second statement, we start with
\begin{gather*}
\mathcal{T}_{p^{m}} = \mathcal{T}_{p} \circ \mathcal{T}_{p^{m-1}} - p^{w + \bar w -1} \mathcal{U}_{p^2} \circ \mathcal{T}_{p^{m-2}} ,
\end{gather*}
for $m \geq 2$ and $p$ prime. Consider the three cases $m=2l$, $m=2l-1$ and $m=2l-2$, with $l \geq 2$:
\begin{gather*}
\mathcal{T}_{p^{2l}} = \mathcal{T}_{p} \circ \mathcal{T}_{p^{2l-1}} - p^{w + \bar w -1} \mathcal{U}_{p^2} \circ \mathcal{T}_{p^{2l-2}} ,\\
\mathcal{T}_{p^{2l-1}} = \mathcal{T}_{p} \circ \mathcal{T}_{p^{2l-2}} - p^{w + \bar w -1} \mathcal{U}_{p^2} \circ \mathcal{T}_{p^{2l-3}} ,\\
\mathcal{T}_{p^{2l-2}} = \mathcal{T}_{p} \circ \mathcal{T}_{p^{2l-3}} - p^{w + \bar w -1} \mathcal{U}_{p^2} \circ \mathcal{T}_{p^{2l-4}} .
\end{gather*}
Inserting the second equation into the first, and using Lemma~\ref{l:commUT}, we get
\begin{gather*}
\mathcal{T}_{p^{2l}} =\mathcal{T}_{p} \circ \mathcal{T}_{p} \circ \mathcal{T}_{p^{2l-2}} - p^{w + \bar w -1} \mathcal{U}_{p^2} \circ \mathcal{T}_{p} \circ \mathcal{T}_{p^{2l-3}} - p^{w + \bar w -1} \mathcal{U}_{p^2} \circ \mathcal{T}_{p^{2l-2}}.
\end{gather*}
Then inserting the third equation, using Lemma~\ref{l:commUT} again, we get
\begin{gather*}
\mathcal{T}_{p^{2l}} =\mathcal{T}_{p} \circ \mathcal{T}_{p} \circ \mathcal{T}_{p^{2l-2}} - p^{w + \bar w -1} \mathcal{U}_{p^2} \circ \mathcal{T}_{p^{2l-2}}\\
\hphantom{\mathcal{T}_{p^{2l}} =}{} - p^{2(w + \bar w -1)} \mathcal{U}_{p^2} \circ \mathcal{U}_{p^2} \circ \mathcal{T}_{p^{2l-4}} - p^{w + \bar w -1} \mathcal{U}_{p^2} \circ \mathcal{T}_{p^{2l-2}}.
\end{gather*}
But
\begin{gather*}
\mathcal{T}_{p} \circ \mathcal{T}_{p} = \mathcal{T}_{p^{2}} + p^{w + \bar w -1} \mathcal{U}_{p^2} ,
\end{gather*}
hence we get
\begin{gather*}
\mathcal{T}_{p^{2l}} =\mathcal{T}_{p^2} \circ \mathcal{T}_{p^{2l-2}} - p^{w + \bar w -1} \mathcal{U}_{p^2} \circ \mathcal{T}_{p^{2l-2}} - p^{2(w + \bar w -1)} \mathcal{U}_{p^2} \circ \mathcal{U}_{p^2} \circ \mathcal{T}_{p^{2l-4}} .
\end{gather*}
We now apply the projection operator $\mathcal{P}_{p^{2l}}$ on both sides of the equation. The left-hand side becomes $\mathcal{H}_{p^{2l}}$, and the last two terms on the right-hand side become
\begin{gather*}
- p^{w + \bar w -1} \mathcal{H}_{p^{2l-2}} - p^{2(w + \bar w -1)}\mathcal{H}_{p^{2l-4}} ,
\end{gather*}
using Lemma~\ref{l:inverses}. For the first term on the right-hand side, we get
\begin{align*}
\mathcal{P}_{p^{2l}} \circ \mathcal{T}_{p^2} \circ \mathcal{T}_{p^{2l-2}}& = \mathcal{P}_{p^{2l-2}} \circ \mathcal{H}_{p^2} \circ \left( \mathcal{P}_{p^{2l-2}} \circ \mathcal{U}_{p^{2l-2}} \right) \circ \mathcal{T}_{p^{2l-2}} \\
& = \mathcal{P}_{p^{2l-2}} \circ \mathcal{H}_{p^2} \circ \mathcal{H}_{p^{2l-2}} \circ \mathcal{U}_{p^{2l-2}},
\end{align*}
where we used Lemmas \ref{l:inverses} and \ref{l:commUT}.
\end{proof}

\section{Comparison to the construction of Bruinier and Stein}\label{form2}

In \cite{bs,stein} Bruinier and Stein also construct Hecke operators on vector-valued modular forms of type $\rho_L$ using a different approach. In this section we compare our Hecke operators to theirs. We find an exact match between our Hecke operators and the Bruinier--Stein operators, for the case $r=p^{2l}$, where $p$ is an odd prime and $l$ a positive integer.

\subsection{The construction of Bruinier and Stein}

Let us now summarize the construction of Bruinier and Stein in \cite{bs,stein}. We refer the reader to~\cite{joshi} and \cite{bs,stein} for further details.

In \cite{bs}, Bruinier and Stein first construct Hecke operators $T^{\rm (BS)}_{m^2}$, where $m$ is a positive integer that is coprime with the level $N$ of the lattice $L$, by extending the Weil representation of $\mathrm{Mp}_2(\mathbb{Z})$ on $\mathbb{C}[A]$ to some appropriate subgroup of $\widetilde{\mathrm{GL}}_2^+(\mathbb{Q})$, where $\widetilde{\mathrm{GL}}_2^+(\mathbb{Q})$ denotes the metaplectic twofold cover of $\mathrm{GL}_2^+(\mathbb{Q})$. They then extend their construction of the Hecke operators $T^{\rm (BS)}_{m^2}$ to all positive integers $m$, not necessarily coprime to $N$, by somehow extending the Weil representation to a suitable double coset. They prove that their Hecke operators $T^{\rm (BS)}_{m^2}$ satisfy the relation
\begin{gather*}
T^{\rm (BS)}_{m^2} \circ T^{\rm (BS)}_{n^2} = T^{\rm (BS)}_{m^2 n^2}, \qquad \gcd(m,n)=1.
\end{gather*}
However, they only describe the action of $T^{\rm (BS)}_{m^2}$ on the Fourier coefficients of vector-valued modular forms in the first case.

In \cite{stein} Stein presents a formula (Theorem~5.4) for the action of their Hecke operators $T^{\rm (BS)}_{p^{2l}}$ on the Fourier coefficients of vector-valued modular forms, for any odd prime~$p$ and positive integer~$l$. To this end, Stein provides an explicit calculation of the action of their extension of the Weil representation on the standard basis $\{ e_\lambda \}_{\lambda \in A}$ of $\mathbb{C}[A]$~-- see Proposition~5.1 and Theorem~5.2 of~\cite{stein}.

Let $\alpha =\left(\left(\begin{smallmatrix} p^{2l} & 0 \\ 0 & 1 \end{smallmatrix}\right),1 \right) \in \widetilde{\mathrm{GL}}_2^+(\mathbb{Q})$. Start by defining an extension of the Weil representation to $\alpha$ by
\begin{gather*}
\rho_L^{-1}(\alpha) e_\lambda = e_{p^l \lambda}.
\end{gather*}
Then consider the double coset $\widetilde{\Gamma}(1) \alpha \widetilde{\Gamma}(1)$, where $\widetilde{\Gamma}(1) = \mathrm{Mp}_2(\mathbb{Z})$. The action above can be extended to an action on this double coset by
\begin{gather*}
\rho_L^{-1}(\beta) e_\lambda = \rho_L^{-1} (\gamma') \rho_L^{-1}(\alpha) \rho_L^{-1}(\gamma) e_\lambda,
\end{gather*}
where $\beta = \gamma \alpha \gamma'$ and $\gamma, \gamma' \in \widetilde{\Gamma}(1)$.

\begin{Def}[\cite{bs,stein}]Let $\psi(\tau)$ be a holomorphic vector-valued modular form\footnote{Note that the Hecke operators in \cite{bs,stein} are defined for holomorphic vector-valued modular forms, so in this section we will restrict our construction to holomorphic vector-valued modular forms as well.} of weight $k$ and type $\rho_L$. Denote by
\begin{gather}\label{eq:coset1}
\widetilde{\Gamma}(1) \alpha \widetilde{\Gamma}(1) = \bigcup_i \widetilde{\Gamma}(1) \delta_i
\end{gather}
the left coset decomposition. Then the Bruinier--Stein Hecke operator $T_{p^{2l}}$ is defined by
\begin{gather}\label{eq:BS}
T^{\rm (BS)}_{p^{2l}}[\psi](\tau) = p^{2l (k-1)} \sum_{i} \sum_{\lambda \in A} \phi_{\delta_i}(\tau)^{-2k} \psi_{\lambda}(\delta_i \tau) \rho_L^{-1}(\delta_i) e_\lambda,
\end{gather}
where the first sum is over the left coset representatives $\delta_i$ in \eqref{eq:coset1}. It can be shown that $T^{\rm (BS)}_{p^{2l}}[\psi](\tau)$ is a holomorphic vector-valued modular form of weight $k$ and type $\rho_L$.
\end{Def}

To get an explicit formula for the Hecke transform $T^{\rm (BS)}_{p^{2l}}[\psi](\tau)$, one needs to calculate $\rho_L^{-1}(\delta_i) e_\lambda$ for all left coset representatives $\delta_i$. This is what is done in Proposition 5.1 and Theorem~5.2 of \cite{stein}. First, we notice that the left coset decomposition \eqref{eq:coset1} can be written explicitly as
\begin{gather*}%\label{eq:coset}
\widetilde{\Gamma}(1) \alpha \widetilde{\Gamma}(1)=\widetilde{\Gamma}(1)\alpha \cup \bigcup_{a=1}^{2l-1} \bigcup_{b \in (\mathbb{Z}/p^a \mathbb{Z})^*} \tilde{\Gamma}(1)\beta_{b,a} \cup \bigcup_{b \in \mathbb{Z}/p^{2l} \mathbb{Z}} \widetilde{\Gamma}(1) \gamma_b,
\end{gather*}
where
\begin{gather*}
\beta_{b,a} = \left( \begin{pmatrix} p^{2l-a} & b \\ 0 & p^a \end{pmatrix}, \sqrt{p^a} \right), \qquad \gamma_b = \left( \begin{pmatrix} 1 & b \\ 0 & p^{2l} \end{pmatrix}, p^l \right).
\end{gather*}
The extension of the Weil representation to these representatives goes as follows.

First, we get the following result for the $\alpha$ and $\gamma_b$ cosets:
\begin{Prop}[{\cite[Proposition 5.1]{stein}}] Let $p$ be an odd prime, $a$, $l$ positive integers with $a < 2 l$ and $b \in (\mathbb{Z}/p^a \mathbb{Z})^*$. Then
\begin{gather*}
\rho_L^{-1}(\alpha) e_\lambda = e_{p^l \lambda},\\
\rho_L^{-1}(\gamma_b) e_\lambda = \sum_{\substack{\nu \in A \\ p^l \nu =\lambda}} \e(- b q(\nu)) e_\nu.
\end{gather*}
\end{Prop}

We also need the extension of the Weil representation for the $\beta_{b,a}$ cosets. Theorem~5.2 of~\cite{stein} presents explicit formulae for the cases $l \geq a$ and $l < a$ separately. However, we claim that there is a mistake in the calculation leading to the formula for the case $l<a$ presented in Theorem~5.2 of~\cite{stein}. As such, we provide here new formulae for the extension of the Weil representation studied in~\cite{bs,stein}. The formula for the case $l \geq a$ that we obtain is equivalent to the one presented in Theorem~5.2 of~\cite{stein}, but our formula for the case $l<a$ is not. For completeness, we provide a derivation of these formulae in Appendix~\ref{a:derivation}.

\begin{restatable}{Prop}{steincorrect}\label{p:steincorrect}
Let $p$ be an odd prime, $a$, $l$ positive integers with $a < 2 l$ and $b \in (\mathbb{Z}/p^a \mathbb{Z})^*$. Then
\begin{gather*}
\rho_L^{-1}( \beta_{b,a}) e_\lambda =
\begin{cases}
\displaystyle p^{-\frac{a}{2} \dim(L)} \sum_{\substack{\delta \in A(p^a) \\ p^a \delta = \lambda}} \e\left( - b q_{p^a}(\delta) \right) e_{p^{l-a} \lambda} & \text{if $l \geq a$,} \\
\displaystyle p^{-\frac{a}{2} \dim(L)} \sum_{\substack{\mu \in A \\ p^{a-l} \mu = \lambda}} \sum_{\substack{\delta \in A(p^l) \\ p^l \delta = \mu}} \e \big({-} b p^{a-l} q_{p^l}(\delta) \big) e_\mu & \text{if $l < a$.}
\end{cases}
\end{gather*}
\end{restatable}

Substituting these two Propositions in \eqref{eq:BS} and simplifying, we can write
\begin{gather*}
T^{\rm (BS)}_{p^{2l}}[\psi](\tau) = C^{\rm (BS)}_{\alpha}(\tau) + \sum_{a=1}^{2l-1} \sum_{b \in (\mathbb{Z}/p^a \mathbb{Z})^*} C^{\rm (BS)}_{\beta_{b,a}}(\tau) + \sum_{b =0}^{p^{2l}-1} C^{\rm (BS)}_{\gamma_b}(\tau),
\end{gather*}
with
\begin{gather}\label{eq:CBS1}
C^{\rm (BS)}_{\alpha}(\tau) = p^{2l (k-1)} \sum_{\lambda \in A} \psi_{\lambda}\big( p^{2l} \tau \big) e_{p^l \lambda},\\
\label{eq:CBS2}
C^{\rm (BS)}_{\beta_{b,a}}(\tau) = \begin{cases}
\displaystyle
p^{(2l-a) k - 2l - \frac{a}{2} \dim(L)} \sum_{\lambda \in A} \sum_{\substack{\delta \in A(p^a) \\ p^a \delta = \lambda}} \e\left( - b q_{p^a}(\delta) \right) & \\
\qquad {}\times \psi_{\lambda}\left( \frac{p^{2l-a} \tau + b}{p^a} \right) e_{p^{l-a} \lambda} & \text{if $l \geq a$,}\vspace{1mm}\\
\displaystyle
p^{(2l-a) k - 2l- \frac{a}{2} \dim(L)}
\sum_{\rho \in A } \sum_{\substack{\delta \in A(p^l) \\ p^l \delta = \rho}} \e \big({-} b p^{a-l} q_{p^l}(\delta) \big) & \\
\qquad {}\times \psi_{p^{a-l} \rho} \left( \frac{p^{2l-a} \tau + b}{p^a} \right) e_\rho
& \text{if $l < a$,}
\end{cases}\\
\label{eq:CBS3}
C^{\rm (BS)}_{\gamma_b}(\tau) = p^{- 2 l } \sum_{\lambda \in A} \e(- b q(\lambda))\psi_{p^l \lambda}\left( \frac{\tau +b}{p^{2 l} }\right) e_\lambda .
\end{gather}

\subsection{Comparison}

Let us now compare the Bruinier--Stein construction with our Hecke operators $\mathcal{H}_{p^{2 l}}$. Since the Bruinier--Stein operators are defined for holomorphic vector-valued modular forms, we also restrict our construction to holomorphic vector-valued modular forms. In particular, we only consider weights of the form $(v, \bar v) = (k,0)$.

From Proposition \ref{p:Ytau}, we can write
\begin{gather*}
\mathcal{H}_{p^{2l}}[\psi](\tau)
=p^{2l (k-1)} \sum_{\lambda \in A} \Bigg(\sum_{\substack{\gamma \in A(p^{2l})\\ p^l \gamma = \lambda}} \sum_{a=0}^{2l} \frac{1}{p^{a \left(k + \frac{1}{2} \dim(L) \right)}} \\
\hphantom{\mathcal{H}_{p^{2l}}[\psi](\tau)=}{} \times \sum_{b=0}^{p^{a}-1} \Delta_{p^{2l}}\big(\gamma,p^l\big) \Delta_{p^{2l}}\big(\gamma,p^{2l-a}\big) \e\left( - \frac{b}{p^{2l-a}} q_{p^{2l}} (\gamma) \right) \psi_{p^{a} \gamma} \left( \frac{p^{2l-a} \tau + b}{p^{a}} \right) \Bigg) e_\lambda.
\end{gather*}
In fact, as mentioned in Remark \ref{r:heckecoprime}, we will consider a slightly different definition of the Hecke transform in this section. We use a double coset definition as in Bruinier and Stein, which simply amounts to restricting the sum over~$b$ to those that are in~$(\mathbb{Z}/p^a \mathbb{Z})^*$.

To compare with the Hecke operator of Bruinier and Stein, we define
\begin{gather*}
\mathcal{H}_{p^{2l}}[\psi](\tau) = C_\alpha(\tau) + \sum_{a=1}^{2l-1} \sum_{b \in (\mathbb{Z}/p^a \mathbb{Z})^*} C_{\beta_{b,a}}(\tau) + \sum_{b =0}^{p^{2l}-1} C_{\gamma_b}(\tau),
\end{gather*}
with
\begin{gather*}
C_\alpha(\tau) = p^{2l (k-1)} \sum_{\lambda \in A} \sum_{\substack{\gamma \in A(p^{2l})\\ p^l \gamma = \lambda}} \Delta_{p^{2l}}\big(\gamma,p^l\big) \Delta_{p^{2l}}\big(\gamma,p^{2l}\big) \psi_{ \gamma} \big( p^{2l} \tau \big) e_\lambda,\\
C_{\beta_{b,a}}(\tau) = p^{(2l-a) k-2l -\frac{a}{2} \dim(L) } \\
\hphantom{C_{\beta_{b,a}}(\tau) =}{} \times \sum_{\lambda \in A}\sum_{\substack{\gamma \in A(p^{2l})\\ p^l \gamma = \lambda}} \Delta_{p^{2l}}\big(\gamma,p^l\big) \Delta_{p^{2l}}\big(\gamma,p^{2l-a}\big) \e\!\left( - \frac{b}{p^{2l-a}} q_{p^{2l}} (\gamma) \right) \psi_{p^{a} \gamma}\! \left( \frac{p^{2l-a} \tau + b}{p^{a}} \right) e_\lambda, \\
C_{\gamma_b}(\tau) = p^{-2l - l \dim(L)} \sum_{\lambda \in A}\sum_{\substack{\gamma \in A(p^{2l})\\ p^l \gamma = \lambda}} \Delta_{p^{2l}}\big(\gamma,p^l\big) \Delta_{p^{2l}}(\gamma,1) \e\big( {-} b q_{p^{2l}} (\gamma) \big) \psi_{p^{2l} \gamma} \left( \frac{ \tau + b}{p^{2l}} \right) e_\lambda.
\end{gather*}
We need to compare these three expressions with the corresponding expressions obtained by Bruinier and Stein in equations~\eqref{eq:CBS1}, \eqref{eq:CBS2} and~\eqref{eq:CBS3}.

\subsubsection[The $\alpha$ coset]{The $\boldsymbol{\alpha}$ coset}

We have
\begin{gather*}
C_\alpha(\tau) = p^{2l (k-1)}\sum_{\lambda \in A} \sum_{\substack{\gamma \in A(p^{2l})\\ p^l \gamma = \lambda}} \Delta_{p^{2l}}\big(\gamma,p^l\big) \Delta_{p^{2l}}\big(\gamma,p^{2l}\big) \psi_{\gamma} \big( p^{2l} \tau \big) e_\lambda.
\end{gather*}
The $\Delta_{p^{2l}}\big(\gamma,p^{2l}\big)$ implies that $\gamma \in A$, in which case the condition $\Delta_{p^{2l}}\big(\gamma,p^l\big)$ becomes trivial. So we can rewrite this expression as
\begin{gather*}
C_\alpha(\tau) = p^{2l (k-1)} \sum_{\lambda \in A} \psi_{ \lambda} \big( p^{2l} \tau \big) e_{p^l \lambda},
\end{gather*}
which is precisely $C_{\alpha}^{\rm (BS)}(\tau)$ in \eqref{eq:CBS1}.

\subsubsection[The $\gamma_b$ cosets]{The $\boldsymbol{\gamma_b}$ cosets}

Let us move on to the $\gamma_b$ cosets. We have
\begin{gather*}
C_{\gamma_b}(\tau) = p^{-2l - l \dim(L)} \sum_{\lambda \in A}\sum_{\substack{\gamma \in A(p^{2l})\\ p^l \gamma = \lambda}} \Delta_{p^{2l}}\big(\gamma,p^l\big) \Delta_{p^{2l}}(\gamma,1) \e\big( {-} b q_{p^{2l}} (\gamma) \big) \psi_{p^{2l} \gamma} \left( \frac{ \tau + b}{p^{2l}} \right) e_\lambda.
\end{gather*}
The $\Delta_{p^{2l}}(\gamma,1)$ condition is trivial, while the $ \Delta_{p^{2l}}(\gamma,p^l)$ condition imposes that $\gamma \in A\big(p^l\big) \subseteq A\big(p^{2l}\big)$. So we can write
\begin{align*}
C_{\gamma_b}(\tau) & = p^{-2l - l \dim(L)} \sum_{\lambda \in A}\sum_{\substack{\gamma \in A(p^{l})\\ p^l \gamma = \lambda}} \e\big({-} b q \big( p^l \gamma \big) \big) \psi_{p^{2l} \gamma} \left( \frac{ \tau + b}{p^{2l}} \right) e_{p^l \gamma} \\
& = p^{-2l - l \dim(L)} \sum_{\lambda \in A} \e ( - b q (\lambda) ) \psi_{p^l \lambda} \left( \frac{ \tau + b}{p^{2l}} \right) e_{\lambda} \Bigg(\sum_{\substack{\gamma \in A(p^{l})\\ p^l \gamma = \lambda}} 1 \Bigg) \\
& =p^{-2l } \sum_{\lambda \in A} \e ( - b q (\lambda) ) \psi_{p^l \lambda} \left( \frac{ \tau + b}{p^{2l}} \right) e_{\lambda},
\end{align*}
where we evaluated the summation as in \eqref{eq:usefulsum}. This is precisely $C_{\gamma_b}^{\rm (BS)}(\tau)$ in \eqref{eq:CBS3}.

\subsubsection[The $\beta_{b,a}$ cosets]{The $\boldsymbol{\beta_{b,a}}$ cosets}

The remaining cases correspond to the $\beta_{b,a}$ cosets. We have
\begin{gather*}
C_{\beta_{b,a}}(\tau) = p^{(2l-a) k-2l -\frac{a}{2} \dim(L) )} \\
\hphantom{C_{\beta_{b,a}}(\tau) =}{} \times \sum_{\lambda \in A}\sum_{\substack{\gamma \in A(p^{2l})\\ p^l \gamma = \lambda}} \Delta_{p^{2l}}\big(\gamma,p^l\big) \Delta_{p^{2l}}\big(\gamma,p^{2l-a}\big) \e\!\left( - \frac{b}{p^{2l-a}} q_{p^{2l}} (\gamma) \right) \psi_{p^{a} \gamma} \!\left( \frac{p^{2l-a} \tau + b}{p^{a}} \right) e_\lambda.
\end{gather*}
The $\Delta_{p^{2l}}\big(\gamma,p^l\big)$ condition imposes that $\gamma \in A(p^l) \subseteq A\big(p^{2l}\big)$. We rewrite
\begin{gather*}
C_{\beta_{b,a}}(\tau) = p^{(2l-a) k-2l -\frac{a}{2} \dim(L) } \sum_{\lambda \in A}\sum_{\substack{\gamma \in A(p^{l})\\ p^l \gamma = \lambda}} \Delta_{p^{2l}}\big(\gamma,p^{2l-a}\big)\\
\hphantom{C_{\beta_{b,a}}(\tau) =}{}\times \e\left( - \frac{b}{p^{2l-a}} q_{p^{2l}} (\gamma) \right) \psi_{p^{a} \gamma} \left( \frac{p^{2l-a} \tau + b}{p^{a}} \right) e_\lambda \\
\hphantom{C_{\beta_{b,a}}(\tau)}{}= p^{(2l-a) k-2l -\frac{a}{2} \dim(L) } \sum_{\gamma \in A(p^{l})} \Delta_{p^{2l}}\big(\gamma,p^{2l-a}\big) \\
\hphantom{C_{\beta_{b,a}}(\tau) =}{}\times \e\left( - \frac{b}{p^{2l-a}} q_{p^{2l}} (\gamma) \right) \psi_{p^{a} \gamma} \left( \frac{p^{2l-a} \tau + b}{p^{a}} \right) e_{p^l \gamma}.
\end{gather*}

{\bf $\boldsymbol{l \geq a}$.} Let us consider first the case when $l \geq a$. The $\Delta_{p^{2l}}\big(\gamma,p^{2l-a}\big)$ condition then imposes that $\gamma \in A(p^a)$. Since $l \geq a$, $A(p^a) \subseteq A\big(p^l\big)$, hence we can write
\begin{align*}
C_{\beta_{b,a}}(\tau) &= p^{(2l-a) k-2l -\frac{a}{2} \dim(L) } \sum_{\gamma \in A(p^{a})} \e ( -b q_{p^{a}} (\gamma) ) \psi_{p^{a} \gamma} \left( \frac{p^{2l-a} \tau + b}{p^{a}} \right) e_{p^l \gamma} \\
&= p^{(2l-a) k-2l -\frac{a}{2} \dim(L) } \sum_{\lambda \in A} \Bigg( \sum_{\substack{\gamma \in A(p^{a}) \\ p^a \gamma = \lambda}} \e ( -b q_{p^{a}} (\gamma) ) \Bigg) \psi_{\lambda} \left( \frac{p^{2l-a} \tau + b}{p^{a}} \right) e_{p^{l-a} \lambda}.
%\label{eq:fun1}
\end{align*}
This is precisely $C^{\rm (BS)}_{\beta_{b,a}}(\tau) $ with $l \geq a$ in \eqref{eq:CBS2}.

{\bf $\boldsymbol{l < a}$.} We start with
\begin{gather*}
C_{\beta_{b,a}}(\tau) = p^{(2l-a) k-2l -\frac{a}{2} \dim(L) } \sum_{\lambda \in A}\sum_{\substack{\gamma \in A(p^{l})\\ p^l \gamma = \lambda}} \Delta_{p^{2l}}\big(\gamma,p^{2l-a}\big)\\
\hphantom{C_{\beta_{b,a}}(\tau) =}{}\times \e\left( - \frac{b}{p^{2l-a}} q_{p^{2l}} (\gamma) \right) \psi_{p^{a} \gamma} \left( \frac{p^{2l-a} \tau + b}{p^{a}} \right) e_\lambda .
\end{gather*}
The $\Delta_{p^{2l}}\big(\gamma,p^{2l-a}\big)$ condition imposes that $\gamma \in A(p^a)$, but since $l<a$, $A\big(p^l\big) \subseteq A(p^a)$, hence this condition is vacuous. Thus we get
\begin{gather*}%\label{eq:step}
C_{\beta_{b,a}}(\tau) = p^{(2l-a) k-2l -\frac{a}{2} \dim(L) }\sum_{\lambda \in A} \Bigg( \sum_{\substack{\gamma \in A(p^{l})\\ p^l \gamma = \lambda}} \e\big( {-}b p^{a-l} q_{p^{l}} (\gamma) \big) \Bigg) \psi_{p^{a-l} \lambda} \left( \frac{p^{2l-a} \tau + b}{p^{a}} \right) e_\lambda .
\end{gather*}
This is precisely $C^{\rm (BS)}_{\beta_{b,a}}(\tau) $ with $l < a$ in~\eqref{eq:CBS2}.

\begin{Rem}It was proved in \cite[Theorem~5.6]{bs} that the operators $T^{\rm (BS)}_{n^2}$ preserve cusp forms and are self-adjoint with respect to the Petersson inner product defined in~\eqref{e:petersson}. From the equivalence of the two constructions it follows that the operators $\mathcal{H}_{n^2}$ are self-adjoint as well.
\end{Rem}

\begin{Rem}The action of the operators $T^{\rm (BS)}_{p^{2l}}$ for $l \in \mathbb{N}$ on Fourier coefficients of a holomorphic vector-valued modular form was computed in~\cite{stein}. A similar computation could be done for the operators $\mathcal{T}_{p^l}$ by applying the explicit Definition~\ref{d:hecke1} on holomorphic vector-valued modular forms with a given $q$-expansion and simplifying the series expansion.
\end{Rem}

\appendix

\section{Derivation of Proposition \ref{p:steincorrect}}\label{a:derivation}

In this appendix we derive explicit formulae for the extension of the Weil representation studied in \cite{bs,stein} for the $\beta_{b,a}$ cosets. More precisely, we prove:

\steincorrect*

As noted in the main text, the formula above for the $l \geq a$ case is equivalent to the formula presented in Theorem~5.2 of \cite{stein}. However, the formula for $l < a$ is not. We believe that there is a mistake in the calculation of \cite{stein} for the case $l<a$.

\begin{proof}We follow the beginning of the proof of Theorem~5.2 in \cite{stein}. Note that $h$ and $s$ in \cite{stein} are denoted by $b$ and $a$ respectively in the current paper. Our starting point is equation~(5.9) in \cite{stein}, which in our notation reads
\begin{gather*}
\rho_L^{-1}( \beta_{b,a}) e_\lambda = \frac{1}{\sqrt{|A|^3} \sqrt{|A(p^a)|}} \sum_{\nu, \rho, \mu \in A} \e\big( b r q(\lambda) - p^{2l-a} t q(\nu) - (\nu, \rho) - b (\mu,\lambda) + p^l (\mu, \nu) \big) e_\rho \\
\hphantom{\rho_L^{-1}( \beta_{b,a}) e_\lambda =}{}\times \sum_{\delta \in A(p^a)} \e \big( t q_{p^a}(\delta) - r(\delta, \lambda)_{p^a} + (\mu, \delta)_{p^a} \big).
\end{gather*}
Since
\begin{gather*}
\frac{\sqrt{|A|}}{\sqrt{|A(p^a)|}} = p^{- \frac{a}{2} \dim(L)},
\end{gather*}
we can rewrite this equation as
\begin{gather*}
\rho_L^{-1}( \beta_{b,a}) e_\lambda = \frac{p^{- \frac{a}{2} \dim(L)}}{|A|^2} \sum_{\nu, \rho, \mu \in A} \e\big( b r q(\lambda) - p^{2l-a} t q(\nu) - (\nu, \rho) - b (\mu,\lambda) + p^l (\mu, \nu) \big) e_\rho \\
\hphantom{\rho_L^{-1}( \beta_{b,a}) e_\lambda =}{} \times \sum_{\delta \in A(p^a)} \e \big( t q_{p^a}(\delta) - r(\delta, \lambda)_{p^a} + (\mu, \delta)_{p^a} \big).
\end{gather*}
Note that the integers $r$, $p^a$, $b$ and $t$ are related by
\begin{gather*}
r p^a - b t = 1.
\end{gather*}
In particular, $b$ and $p^a$ are coprime.

{\bf $\boldsymbol{l \geq a}$.} We first consider the case $l \geq a$.
Let us do a shift $\delta \mapsto \delta - p^{l-a} \nu$. Since $p^{l-a} \nu \in A \subseteq A(p^a)$, the shift does not change the sum over $\delta$ since it is just relabeling. We get
\begin{gather*}
\rho_L^{-1}( \beta_{b,a}) e_\lambda = \frac{p^{- \frac{a}{2} \dim(L)}}{|A|^2} \sum_{\nu, \rho, \mu \in A} \e\big( b r q(\lambda) - (\nu, \rho) - b (\mu,\lambda) + r p^l (\nu, \lambda) \big) e_\rho \\
\hphantom{\rho_L^{-1}( \beta_{b,a}) e_\lambda =}{} \times \sum_{\delta \in A(p^a)} \e \big( t q_{p^a}(\delta) - (p^a \delta, t p^{l-a}\nu + r \lambda - \mu) \big).
\end{gather*}
Let us do a further shift $\mu \mapsto \mu + r \lambda + t p^{l-a} \nu$, which also does not change the sum
\begin{gather*}
\rho_L^{-1}( \beta_{b,a}) e_\lambda = \frac{p^{- \frac{a}{2} \dim(L)}}{|A|^2} \sum_{\nu, \rho, \mu \in A} \e\big( {-} b r q(\lambda) - (\nu, \rho) - b (\mu,\lambda) + p^{l-a} (\nu, \lambda) \big) e_\rho \\
\hphantom{\rho_L^{-1}( \beta_{b,a}) e_\lambda =}{} \times \sum_{\delta \in A(p^a)} \e \big( t q_{p^a}(\delta) + (p^a \delta, \mu) \big).
\end{gather*}
The sum over $\nu \in A$ is non-zero and equal to $|A|$ if and only if $\rho = p^{l-a} \lambda$ mod $L$. Thus we get
\begin{gather*}
\rho_L^{-1}( \beta_{b,a}) e_\lambda = \frac{p^{- \frac{a}{2} \dim(L)}}{|A|} \sum_{\mu \in A} \e ( - b r q(\lambda) - b (\mu,\lambda) ) e_{p^{l-a} \lambda} \sum_{\delta \in A(p^a)} \e \big( t q_{p^a}(\delta) + (p^a \delta, \mu) \big).
\end{gather*}
The sum over $\mu$ is then non-zero and equal to $|A|$ if and only if $b \lambda = p^a \delta$ mod $L$. Thus we get
\begin{gather}\label{eq:sss}
\rho_L^{-1}( \beta_{b,a}) e_\lambda =p^{- \frac{a}{2} \dim(L)} \e ( -b r q(\lambda) ) e_{p^{l-a} \lambda}
\sum_{\substack{\delta \in A(p^a) \\ p^a \delta = b \lambda}}\e ( t q_{p^a}(\delta) ).
\end{gather}

Now let $S$ be the set of $\delta \in A(p^a)$ such that $p^a \delta = b \lambda$ for some fixed $\lambda \in A$, and $S'$ be the set of $\delta' \in A(p^a)$ such that $p^a \delta' = \lambda$. We claim that there is a bijection $f\colon S' \to S$ given by $f\colon \delta' \mapsto \delta = b \delta'$.

First, let us show that it is injective. Any two $\delta_1', \delta_2' \in S'$ must differ by an element of $\frac{1}{p^a} L/L \subseteq A(p^a)$, that is, $\delta_1' = \delta_2' + \mu$ for some $\mu \in \frac{1}{p^a }L/L$. But then, $b \delta_1' = b \delta_2' + b \mu$, and $b \mu = 0$ mod $L$ if and only if $\mu = 0$ mod~$L$, since $b$ and $p^a$ are coprime. Therefore $b \delta_1' = b \delta_2'$ if and only if $\delta_1' = \delta_2'$ mod $L$.

Second, we show that $f$ is surjective. We need to show that any $\delta \in S$ can be written as $\delta = b \delta'$ for some $\delta' \in S'$. Pick a $\delta' \in S'$. $\delta$ can be written as $\delta = b \delta' + \mu$ for some $\mu \in A(p^a)$. But then
\begin{gather*}
b \lambda = p^a \delta = b p^a \delta' + p^a \mu = b \lambda + p^a \mu,
\end{gather*}
and hence $p^a \mu = 0$ mod $L$, that is, $\mu \in \frac{1}{p^a} L/L \subseteq A(p^a)$. Now, since $b$ is coprime with $p^a$, we can always write $\mu = b \nu$ for some $\nu \in \frac{1}{p^a} L/L$. Thus we get
\begin{gather*}
\delta = b \delta ' + b \nu = b (\delta' + \nu) = b \delta'',
\end{gather*}
where $\delta'' = \delta' + \nu \in A(p^a)$, and $p^a \delta'' = p^a \delta' + p^a \nu = p^a \delta' = \lambda$. Thus we conclude that $\delta = h \delta''$, with $\delta'' \in S'$.

As a result, the bijection $f\colon S' \to S$ allows us to substitute $\delta = b \delta'$ in \eqref{eq:sss} and replace the sum over $\delta \in A(p^a)$ such that $p^a \delta = b \lambda$ by a sum over $\delta' \in A(p^a)$ such that $p^a \delta' = \lambda$. We get
\begin{gather*}
\rho_L^{-1}( \beta_{b,a}) e_\lambda =p^{- \frac{a}{2} \dim(L)} \e ( -b r q(\lambda) ) e_{p^{l-a} \lambda}
\sum_{\substack{\delta' \in A(p^a) \\ p^a \delta' = \lambda}}\e \big( t b^2 q_{p^a}(\delta') \big).
\end{gather*}
Now using $bt=rp^a -1$,
\begin{align*}
\rho_L^{-1}( \beta_{b,a}) e_\lambda &=p^{- \frac{a}{2} \dim(L)} \e ( -b r q(\lambda) ) e_{p^{l-a} \lambda}
\sum_{\substack{\delta' \in A(p^a) \\ p^a \delta' = \lambda}}\e ( -b q_{p^a}(\delta')) \e (b r q(p^a \delta') ) \\
&=p^{- \frac{a}{2} \dim(L)}
\sum_{\substack{\delta' \in A(p^a) \\ p^a \delta' = \lambda}}\e ( -b q_{p^a}(\delta') ) e_{p^{l-a} \lambda} .
\end{align*}

{\bf $\boldsymbol{l < a}$.} Let us start again with
\begin{gather*}
\rho_L^{-1}( \beta_{b,a}) e_\lambda = \frac{p^{- \frac{a}{2} \dim(L)}}{|A|^2}\sum_{\nu, \rho, \mu \in A} \e\big( b r q(\lambda) - p^{2l-a} t q(\nu) - (\nu, \rho) - b (\mu,\lambda) + p^l (\mu, \nu) \big) e_\rho \\
\hphantom{\rho_L^{-1}( \beta_{b,a}) e_\lambda =}{} \times \sum_{\delta \in A(p^a)} \e \big( t q_{p^a}(\delta) - r(\delta, \lambda)_{p^a} + (\mu, \delta)_{p^a} \big).
\end{gather*}
Let us rewrite the sum over $\nu \in A$ as a sum over $\gamma \in A(p^a)$, with $p^a \gamma = \nu$. This map is not one-to-one; its kernel is given by $ \frac{1}{p^a} L/L \subseteq A(p^a)$. Thus we need to divide by $\big| \frac{1}{p^a} L/L \big| = p^{a \dim(L)}$. We get
\begin{gather*}
\rho_L^{-1}( \beta_{b,a}) e_\lambda= \frac{p^{- \frac{3 a}{2} \dim(L)}}{|A|^2}\\
\hphantom{\rho_L^{-1}( \beta_{b,a}) e_\lambda=}{}\times \sum_{ \rho, \mu \in A} \sum_{\gamma \in A(p^a)} \e\big( b r q(\lambda) - p^{2l} t q_{p^a}(\gamma) - (p^a \gamma, \rho) - b (\mu,\lambda) + p^l (\mu, p^a \gamma) \big) e_\rho \\
\hphantom{\rho_L^{-1}( \beta_{b,a}) e_\lambda=}{}\times \sum_{\delta \in A(p^a)} \e \big( t q_{p^a}(\delta) - (p^a \delta, r \lambda - \mu) \big).
\end{gather*}
We then do a shift $\delta \mapsto \delta - p^l \gamma$. Since $p^l \gamma \in A\big(p^{a-l}\big) \subseteq A(p^a)$, the shift does not change the sum. We get
\begin{gather*}
\rho_L^{-1}( \beta_{b,a}) e_\lambda= \frac{p^{- \frac{3 a}{2} \dim(L)}}{|A|^2} \sum_{\rho, \mu \in A} \sum_{\gamma \in A(p^a)} \e\big( b r q(\lambda) - (p^a \gamma, \rho) - b (\mu,\lambda) + r p^l (p^{a} \gamma, \lambda) \big) e_\rho \\
\hphantom{\rho_L^{-1}( \beta_{b,a}) e_\lambda=}{} \times \sum_{\delta \in A(p^a)} \e \big( t q_{p^a}(\delta) - (p^a \delta, r \lambda - \mu) - t p^l (\delta, \gamma)_{p^a} \big).
\end{gather*}
We now do a shift $\mu \mapsto \mu + r \lambda $ to get
\begin{gather*}
\rho_L^{-1}( \beta_{b,a}) e_\lambda= \frac{p^{- \frac{3 a}{2} \dim(L)}}{|A|^2} \sum_{\rho, \mu \in A} \sum_{\gamma \in A(p^a)} \e\big( {-} b r q(\lambda) - ( \gamma, \rho)_{p^a} - b (\mu,\lambda) + r p^l ( \gamma, \lambda)_{p^a} \big) e_\rho \\
\hphantom{\rho_L^{-1}( \beta_{b,a}) e_\lambda=}{} \times \sum_{\delta \in A(p^a)} \e \big( t q_{p^a}(\delta) + (p^a \delta, \mu) - t p^l (\delta, \gamma)_{p^a} \big).
\end{gather*}
The sum over $ \mu \in A$ is non-zero and equal to $|A|$ if and only if $p^a \delta = b \lambda$ mod $L$. Moreover, as we have seen in the calculation for the $l \geq a$ case, we can substitute $\delta = b \delta'$ and replace the sum over $\delta \in A(p^a)$ such that $p^a \delta = b \lambda$ by a sum over $\delta' \in A(p^a)$ such that $p^a \delta' = \lambda$. Thus we can write
\begin{gather*}
\rho_L^{-1}( \beta_{b,a}) e_\lambda= \frac{p^{- \frac{3 a}{2} \dim(L)}}{|A|} \sum_{\rho \in A} \sum_{\gamma \in A(p^a)} \e\big( {-} b r q(\lambda) - ( \gamma, \rho)_{p^a} + r p^l ( \gamma, \lambda)_{p^a} \big) e_\rho \\
\hphantom{\rho_L^{-1}( \beta_{b,a}) e_\lambda=}{} \times \sum_{\substack{\delta' \in A(p^a) \\ p^a \delta' = \lambda}} \e \big( t b^2 q_{p^a}(\delta') - b t p^l (\delta', \gamma)_{p^a} \big).
\end{gather*}
Using $b t=r p^a -1$,
\begin{gather*}
\rho_L^{-1}( \beta_{b,a}) e_\lambda= \frac{p^{- \frac{3 a}{2} \dim(L)}}{|A|} \sum_{\rho \in A} \sum_{\gamma \in A(p^a)} \e\left( - ( \gamma, \rho)_{p^a} \right) e_\rho \sum_{\substack{\delta' \in A(p^a) \\ p^a \delta' = \lambda}} \e \big( {-} b q_{p^a}(\delta') + p^l (\delta', \gamma)_{p^a} \big).
\end{gather*}
The sum over $\gamma \in A(p^a)$ is non-zero and equal to $|A(p^a)|$ if and only if $\rho = p^l \delta'$ mod $L$. In particular, since $\rho \in A$, this implies that the sum can be non-zero only for $\delta' \in A\big(p^l\big) \subseteq A(p^a)$. Thus we get
\begin{align*}
\rho_L^{-1}( \beta_{b,a}) e_\lambda& = \frac{p^{- \frac{3 a}{2} \dim(L)} |A(p^a)|}{|A|}
\sum_{\substack{\delta' \in A(p^l) \\ p^a \delta' = \lambda}} \e \big( {-} b p^{a-l} q_{p^l}(\delta') \big) e_{p^l \delta'} \\
& = p^{- \frac{ a}{2} \dim(L)} \sum_{\substack{\mu \in A \\ p^{a-l} \mu = \lambda}}
\sum_{\substack{\delta' \in A(p^l) \\ p^l \delta' = \mu}} \e \big( {-} b p^{a-l} q_{p^l}(\delta') \big) e_{\mu}.\tag*{\qed}
\end{align*}\renewcommand{\qed}{}
\end{proof}

\subsection*{Acknowledgements}

We would like to thank Duiliu-Emanuel Diaconescu for interesting discussions and collaboration at the initial stages of this project. We would also like to thank Terry Gannon, Jeff Harvey and Martin Raum for useful discussions. Finally, we would like to thank the anonymous referees for their very valuable comments. We acknowledge the support of the Natural Sciences and Engineering Research Council of Canada.

\pdfbookmark[1]{References}{ref}
\LastPageEnding

\end{document}